\documentclass[11pt,reqno,twoside]{amsart}
\usepackage{amssymb,amsmath,amsthm,soul,color,paralist}
\usepackage{t1enc}
\usepackage[cp1250]{inputenc}
\usepackage{a4,indentfirst,latexsym}
\usepackage{graphics}
\usepackage{mathrsfs}
\usepackage{cite,enumitem,graphicx}
\usepackage[colorlinks=true,urlcolor=blue,
citecolor=red,linkcolor=blue,linktocpage,pdfpagelabels,
bookmarksnumbered,bookmarksopen]{hyperref}
\usepackage[english]{babel}
\usepackage[left=2.50cm,right=2.50cm,top=2.72cm,bottom=2.72cm]{geometry}
\usepackage[metapost]{mfpic}
\usepackage[hyperpageref]{backref}
\usepackage[colorinlistoftodos]{todonotes}
\usepackage[normalem]{ulem}

\makeatletter
\providecommand\@dotsep{5}
\def\listtodoname{List of Todos}
\def\listoftodos{\@starttoc{tdo}\listtodoname}
\makeatother

\numberwithin{equation}{section}

\newcommand{\h}{H^{s}_{\e}}
\newcommand{\R}{\mathbb{R}}
\newcommand{\2}{2^{*}_{s}}

\newcommand{\C}{\mathbb{C}}

\newcommand{\N}{\mathcal{N}}

\DeclareMathOperator{\dive}{div}
\DeclareMathOperator{\supp}{supp}
\DeclareMathOperator{\e}{\varepsilon}

\newtheorem{lem}{Lemma}[section]
\newtheorem{thm}{Theorem}[section]

\keywords{Fractional magnetic operators, variational methods, concentration phenomenon}
\subjclass[2010]{35A15, 35R11, 35S05}

\date{}

\begin{document}
\title[Fractional Schr\"odinger-Poisson type equation with magnetic fields]{Existence and concentration of nontrivial solutions for a fractional magnetic Schr\"odinger-Poisson type equation}

\author[V. Ambrosio]{Vincenzo Ambrosio}
\address{Vincenzo Ambrosio\hfill\break\indent 
Department of Mathematics  \hfill\break\indent
EPFL SB CAMA \hfill\break\indent
Station 8 CH-1015 Lausanne, Switzerland}
\email{vincenzo.ambrosio2@unina.it}

\begin{abstract}
%By using variational methods, we investigate the existence and concentration behavior of nontrivial solutions for 
We consider the following fractional Schr\"odinger-Poisson type equation with magnetic fields
\begin{equation*}
\varepsilon^{2s}(-\Delta)_{A/\varepsilon}^{s}u+V(x)u+\e^{-2t}(|x|^{2t-3}*|u|^{2})u=f(|u|^{2})u \quad \mbox{ in } \mathbb{R}^{3},
\end{equation*}
where $\varepsilon>0$ is a parameter, $s\in (\frac{3}{4}, 1)$, $t\in (0,1)$, $(-\Delta)^{s}_{A}$ is the fractional magnetic Laplacian, $A:\mathbb{R}^{3}\rightarrow \mathbb{R}^{3}$ is a smooth magnetic potential, $V:\R^{3}\rightarrow \R$ is a positive continuous electric potential and $f:\mathbb{R}^{3}\rightarrow \mathbb{R}$ is a continuous function with subcritical growth.
By using suitable variational methods, we  show the existence of families of nontrivial solutions concentrating
around local minima of the potential $V(x)$ as $\e\rightarrow 0$.
%Under a local condition on the potential, we prove existence and concentration of solutions as $\e\rightarrow 0$.
\end{abstract}

\maketitle

\section{introduction}
In this paper we are interested in the existence of nontrivial solutions for the following fractional nonlinear Schr\"odinger-Poisson type equation
\begin{equation}\label{P}
\varepsilon^{2s}(-\Delta)_{A/\varepsilon}^{s}u+V(x)u+\e^{-2t}(|x|^{2t-3}*|u|^{2})u=f(|u|^{2}) \quad \mbox{ in } \mathbb{R}^{3},
\end{equation}
where $\e>0$ is a parameter, $s\in (\frac{3}{4}, 1)$, $t\in (0,1)$, $A:\R^{3}\rightarrow \R^{3}\in C^{0, \alpha}$, with $\alpha\in (0, 1]$, is a magnetic potential, and $(-\Delta)^{s}_{A}$ is the so called fractional magnetic Laplacian which can be defined by setting
\begin{equation}\label{operator}
(-\Delta)^{s}_{A}u(x)
:=c_{3,s} \lim_{r\rightarrow 0} \int_{B_{r}^{c}(x)} \frac{u(x)-e^{\imath (x-y)\cdot A(\frac{x+y}{2})} u(y)}{|x-y|^{3+2s}} dy,
\quad
c_{3,s}:=\frac{4^{s}\Gamma\left(\frac{3+2s}{2}\right)}{\pi^{3/2}|\Gamma(-s)|},
\end{equation}
for any $u\in C^{\infty}_{c}(\R^{3}, \C)$; see \cite{DS, I10} for more details.
As showed in \cite{SV} (see also \cite{PSV}), when $s\rightarrow 1$, the previous operator reduces to the magnetic Laplacian $-\Delta_{A}:=\left(\frac{1}{\imath}\nabla-A\right)^{2}$ (see \cite{LaL, LL}) given by
$$
-\Delta_{A} u= -\Delta u -\frac{2}{\imath} A(x) \cdot \nabla u + |A(x)|^{2} u -\frac{1}{\imath} u \dive(A(x)),
$$ 
which appears in the study of the following Schr\"odinger equation with magnetic fields 
\begin{equation}\label{CMSE}
%\label{MSE}
-\Delta_{A} u+V(x)u=f(x, |u|^{2})u \quad \mbox{ in } \R^{N}.
\end{equation}
Equation \eqref{CMSE} has been widely investigated by several authors in the last thirty years;  see for instance \cite{AFF, AS, Cingolani, CSS, EL, K}.

Along the paper, we assume that $V:\R^{3}\rightarrow \R$  is a continuous potential satisfying the following assumptions due to del Pino and Felmer \cite{DF}:
\begin{compactenum}[$(V_1)$]
\item $\inf_{x\in \R^{3}} V(x)=V_{0}>0$;
\item  there exists a bounded domain $\Lambda\subset \R^{3}$ such that
\begin{equation}
V_{0}<\min_{\partial \Lambda} V \quad \mbox{ and } M=\{x\in \Lambda: V(x)=V_{0}\}\neq \emptyset,
\end{equation}
\end{compactenum}
and  $f:\R\rightarrow \R$ is a continuous function verifying the following conditions:
\begin{compactenum}[$(f_1)$]
\item $f(t)=0$ for $t\leq 0$ and $\displaystyle{\lim_{t\rightarrow 0} \frac{f(t)}{t}=0}$;
\item there exist $q\in (4, 2^{*}_{s})$ such that 
$$
\lim_{t\rightarrow \infty} \frac{f(t)}{t^{\frac{q-2}{2}}}=0;
$$
\item there exists $\theta\in (4, \2)$ such that $0<\frac{\theta}{2} F(t)\leq t f(t)$ for any $t>0$, where $F(t)=\int_{0}^{t} f(\tau)d\tau$;
\item $t\mapsto \frac{f(t)}{t}$ is increasing for $t>0$.
\end{compactenum} 

Let us state our main theorem:
\begin{thm}\label{thm1}
Assume that $(V_1)$-$(V_2)$ and $(f_1)$-$(f_4)$ hold. Then
there exists $\e_{0}>0$ such that, for any $\e\in (0, \e_{0})$, problem \eqref{P} has a nontrivial solution. Moreover, if $u_{\e}$ denotes one of these solutions and $x_{\e}$ the global maximum point of $|u_{\e}|$, then we have 
$$
\lim_{\e\rightarrow 0} V(x_{\e})=V_{0}
$$	
and
$$
|u_{\e}(x)|\leq \frac{C\e^{3+2s}}{C\e^{3+2s}+|x-x_{\e}|^{3+2s}} \quad \forall x\in \R^{3}.
$$
\end{thm}

The above result is motivated by some works that appeared in the last years concerning fractional Schr\"odinger equations with magnetic fields of the type
%In the case $A\neq 0$, only few and recent works deal with the following fractional magnetic Schr\"odinger equation
\begin{equation}\label{MSE}
\e^{2s}(-\Delta)^{s}_{A}u+V(x)u=f(x, |u|^{2})u \quad \mbox{ in } \R^{N}.
\end{equation}
For instance, in the unperturbed case (that is $\e=1$), d'Avenia and Squassina \cite{DS} studied via a constrained minimization argument the existence of solutions \eqref{MSE}, $V$ is constant and $f$ is a subcritical or critical nonlinearity. 
Fiscella et al. \cite{FPV} obtained a multiplicity result for a fractional magnetic problem with homogeneous boundary conditions. 
%Mingqi et al. \cite{MRZ} investigated a critical fractional Choquard-Kirchhoff equation with magnetic field.
When $\e>0$ is small, Zhang et al. \cite{ZSZ} considered a fractional magnetic Schr\"odinger equation involving critical frequency and critical growth. 
%proving the existence of mountain pass solutions which tend to the trivial solution as $\e\rightarrow 0$.
%Wang and Xiang \cite{WX} investigated the multiplicity of solutions to a nonlocal fractional Choquard equation involving an external magnetic potential and critical exponent.
Recently, in \cite{AD}, the author and d'Avenia dealt with the existence and the multiplicity of solutions to \eqref{MSE} for small $\e>0$, when the potential $V$ satisfies the global condition due to Rabinowitz \cite{Rab} and $f$ has a subcritical growth.  
%We also mention \cite{ZSZ, MPSZ, NPSV} for other interesting results.

%When $A=0$, the fractional magnetic Laplacian $(-\Delta)^{s}_{A}$ reduces to the fractional Laplacian $(-\Delta)^{s}$ which has achieved a tremendous popularity in these last twenty years due to its great applications in several contexts such as phase transitions, quasi-geostrophic flows, game theory, population dynamics, quantum mechanics and so on; see \cite{AV, DPV, MBRS} for more details.
In absence of a magnetic field (that is $A=0$), the fractional magnetic Laplacian $(-\Delta)^{s}_{A}$ coincides with the fractional Laplacian $(-\Delta)^{s}$ and the equation \eqref{MSE} becomes
%From a mathematical point of view,  several contributions \cite{A1, A0, A3, FQT, FLS, Secchi1} have been given in the study of 
the well-known fractional Schr\"odinger equation (see \cite{Laskin})
\begin{equation}\label{FSE}
\e^{2s}(-\Delta)^{s}u+V(x)u=f(x,u) \mbox{ in } \R^{N},
\end{equation} 
 for which the existence and concentration phenomena of positive solutions have been considered by many mathematicians.
For example, D\'avila et al. \cite{DDPW} used a Lyapunov-Schmidt variational reduction to prove that \eqref{FSE} has a multi-peak solution when $V\in L^{\infty}(\R^{N})\cap C^{1, \alpha}(\R^{N})$ is a positive potential and $f$ is a subcritical nonlinearity; see also \cite{DDPDV} in which a concentration result has been established for a nonlocal problem with Dirichlet datum.
Fall  et al. \cite{FMV} showed that the concentration points of the solutions of \eqref{FSE} must be the critical points for $V$, as $\e$ tends to zero. 
Alves and Miyagaki \cite{AM} (see also \cite{A1, A3}) used the penalization method in \cite{DF} to study the existence and concentration of positive solutions of \eqref{FSE} requiring that $f$ satisfies $(f_1)$-$(f_4)$ and $V$ verifies $(V_1)$-$(V_2)$.

%Such results have been extended for critical and supercritical nonlinearities in \cite{A1, HZ2}.
%For simplicity, we assume that $0\in \Lambda$ and $V_{0}=V(0)$.
%In \cite{A1} the author studied existence and multiplicity of positive solutions to \eqref{FSE} when $f$ has a subcritical, critical or supercritical growth.
On the other hand, in these last years, several authors investigated fractional Schr\"odinger-Poisson systems of the type
\begin{equation}\label{FSPS} 
\left\{
\begin{array}{ll}
\e^{2s}(-\Delta)^{s} u+V(x) \phi u = g(x, u) &\mbox{ in } \R^{3} \\
\e^{2t}(-\Delta)^{t} \phi= u^{2} &\mbox{ in } \R^{3},
\end{array}
\right.
\end{equation}
%which appears in quantum mechanics models (see e.g. [37]) and in semiconductor theory [41]. 
which can be seen as the nonlocal counterpart of the well-known Schr\"odinger-Poisson systems appearing in quantum mechanics models \cite{BBL} and in semiconductor theory \cite{Mar}. Such systems have been introduced in \cite{BF} to 
describe systems of identical charged particles interacting each other in the case that effects of magnetic field could be ignored and its solution represents, in particular, a standing wave for such a system. 
We refer to \cite{AdAP, DW, He, HL, ruiz, WTXZ, ZZ} for some interesting existence and multiplicity results for classical perturbed and unperturbed Schr\"odinger-Poisson systems. 
%without magnetic field. 
%When $A\neq 0$, we refer to \cite{ZS} in which the authors proved the multiplicity and concentration of nontrivial solutions for a magnetic Schr\"odinger-Poisson type equation.

Concerning \eqref{FSPS}, Giammetta \cite{G} considered the local and global well-posedness  of a one dimensional fractional Schr\"odinger-Poisson system in which $\e=1$ and the fractional diffusion appears only in the Poisson equation. %in \eqref{FSPS}.
Zhang et al. \cite{ZDS} dealt with the existence of positive solutions to \eqref{FSPS} with $\e=1$, $V(x)=\mu>0$ and $g$ is a general nonlinearity having subcritical or critical growth.  
Murcia and Siciliano \cite{MS} proved that, for suitably small $\e$, the number of positive solutions to a doubly singularly perturbed fractional Schr\"odinger-Poisson system is estimated below by the Ljusternick-Schnirelmann category of the set of minima of the potential. 
%Teng \cite{teng} combined the Nehari-Pohozaev manifold with the monotone trick and a global compactness Lemma, to investigate the existence of ground state solutions for a critical fractional Schr\"odinger-Poisson system like \eqref{FSPS} with $\e=1$.
Liu and Zhang \cite{LZ} studied multiplicity and concentration of solutions to \eqref{FSPS} involving the critical exponent and under a global condition on the potential $V$.
Teng \cite{Teng}, inspired by \cite{HL}, used the penalization method due to Byeon and Wang \cite{BW} to analyze the existence and concentration of positive solutions to \eqref{FSPS} under the conditions $(V_1)$-$(V_2)$ and $g$ is a $C^{1}$ subcritical nonlinearity. 

%To the best of our knowledge, the fractional magnetic equation \eqref{P} has not ever been considered until now, so the aim of this paper is to fill this gap giving a first result in this direction. More precisely, 
Particularly motivated by \cite{AM, AD, Teng}, in this paper we investigate the existence and concentration behavior of nontrivial solutions to \eqref{P} with $A\neq 0$ and under the assumptions $(V_1)$-$(V_2)$ and $(f_1)$-$(f_4)$. 
We note that when $s=t=1$ in \eqref{P}, the multiplicity and concentration for a Schr\"odinger-Poisson type equation with magnetic field and under a local condition on $V$, has been established in \cite{ZS} by using some ideas developed in \cite{AFF}.
%combined penalization technique and Ljusternik-Schnirelmann theory. 
Anyway, their arguments work for $C^{1}$-Nehari manifolds and we can not apply them in our situation because we are assuming the only continuity of $f$.  
%To our knowledge, this type of study for \eqref{P} is not present in literature.
% and this represents the novelty of our work.

Since we don't have any information on the behavior of $V$ at infinity, we adapt the penalization argument developed by del Pino and Felmer in \cite{DF}, which consists in making an appropriate modification on $f$, solving a modified problem and then check that, for $\e$ small enough, the solutions of the modified problem are indeed solutions of the original one.
We point out that the penalization argument developed here is different from the one used in \cite{Teng}, in which the author does not assume the suplinear-4 growth on $f$ but has to require $f\in C^{1}$ to apply the techniques developed in \cite{BW, HL}.
The existence of nontrivial solutions for the modified problem is obtained by using the Mountain Pass Theorem \cite{AR} to the functional $J_{\e}$ associated to the modified problem.
We note that the main issue in the study of $J_{\e}$ concerns with the verification the Palais-Smale compactness condition. Indeed, the presence of the fractional magnetic Laplacian and the convolution term $(|x|^{2t-3}*|u|^{2})$, make our study more complicated and intriguing, and some suitable arguments will be needed to achieve our purpose; see Lemma \ref{PSc}.
The next step is to show that if $u_{\e}$ is a solution of the modified problem, then $u_{\e}$ is also a solution of the original one \eqref{P}.
In the case $A=0$ (see \cite{AM, Teng}), this is proved taking into account some fundamental estimates established in \cite{FQT} concerning the Bessel operator. In the case $A\neq 0$, we don't have similar informations for the following fractional equation
\begin{equation}\label{FMSEA}
(-\Delta)^{s}_{A}u+V_{0}u=h(|u|^{2})u \mbox{ in } \R^{3}.
\end{equation}
For the above reason, we use a new approximation argument which allows us to deduce that if $u_{\e}$ is  a solution to the modified problem, then $|u_{\e}|$ is a subsolution for an autonomous fractional Schr\"odinger equation without magnetic field, and then we apply a comparison argument to deduce informations on the behavior at infinity of $|u_{\e}|$; see Lemma \ref{moser}. 
We point out that, in the case $s=1$, a similar reasoning works (see \cite{CS, K}) in view of the following distributional Kato's inequality \cite{Kato} 
%, this type of argument is a direct consequence of Kato's inequality \cite{Kato} 
%$$
%-\Delta_{A} u+V_{0}u=h(|u|^{2})u \mbox{ in } \R^{3},
%$$
%then $|u|$ is a subsolution to 
%$$
%-\Delta|u|+V_{0}u=h(|u|^{2})|u| \mbox{ in } \R^{3},
%$$
%in view of the Kato's inequality \cite{Kato} 
$$
-\Delta|u|\leq \Re(sign(u)(-\Delta_{A} u)).
$$ 
%and then we can use standard arguments to prove that $|u(x)|\rightarrow 0$ as $|x|\rightarrow \infty$ (the decay is exponential); see for instance \cite{K}.
Recently, in \cite{HIL}, under the restriction $s\in (0, 1/2]$, a fractional distributional Kato's inequality has been established for  some fractional magnetic operators which also include $(-\Delta)^{1/2}_{A}$. 
We suspect that a fractional Kato's inequality is available in our setting for any fractional power $s$ (indeed it is easily seen that a pointwise Kato's inequality holds for smooth functions), but we are not able to prove it. 
Again, we can not repeat the iteration done in \cite{AFF} to obtain $L^{\infty}$-estimates on the modulus of solutions, due to the nonlocal character of $(-\Delta)^{s}_{A}$.
Anyway, in the present paper, we introduce some new arguments which we believe to be useful to be applied in other situations to obtain $L^{\infty}$-estimates for problems like \eqref{MSE}. 
%This fact represents the novelty of our work.
Now we give a sketch of our idea.
%However, in the nonlocal magnetic framework, we don't know if a similar Kato's inequality is available for $(-\Delta)^{s}_{A}$, and for this reason we develop some new arguments. 
% Moreover, we are not able to adapt in our framework the iteration scheme done in \cite{AFF} in which the authors proved $|u_{n}(x)|\rightarrow 0$ as $|x|\rightarrow \infty$ uniformly with respect to $n\in \mathbb{N}$. For the above reasons, in this work we develop some new ideas that go as follows. 
Firstly, we show that  the (translated) sequence $|u_{n}|$ of solutions of the modified problem is bounded in $L^{\infty}(\R^{3}, \R)$ uniformly in $n\in \mathbb{N}$, by using an appropriate Moser iteration scheme \cite{Moser}. After that, we prove that $|u_{n}|$ verifies
\begin{equation*}
(-\Delta)^{s}|u_{n}|+V_{0}|u_{n}|\leq g(\e x, |u_{n}|^{2})|u_{n}| \mbox{ in } \R^{3},
\end{equation*}
by using $\displaystyle{\frac{u_{n}}{u_{\delta,n}}\varphi}$ as test function in the modified problem, 
%(this function is admissible because $|u_{n}|\in L^{\infty}(\R^{3}, \R)$),  
where $u_{\delta,n}=\sqrt{|u_{n}|^{2}+\delta^{2}}$ and $\varphi$ is a real smooth nonnegative function with compact support in $\R^{3}$, and then we take the limit as $\delta\rightarrow 0$. In some sense, we are going to prove a fractional Kato's inequality for the modified problem.
At this point, by comparison, we can show that $|u_{n}(x)|\rightarrow 0$ as $|x|\rightarrow \infty$ uniformly with respect to $n\in \mathbb{N}$, 
%This  fact plays a very important role because we can deduce informations on the decay of $|u_{n}|$ 
taking into account the power type decay of solutions of autonomous fractional Schr\"odinger equations; see \cite{FQT}.
\noindent

The paper is organized as follows. In Section $2$ we give some results on fractional magnetic Sobolev spaces and we recall some useful lemmas. In Section $3$, we introduce the modified problem and we show that the corresponding functional satisfies the assumptions of the Mountain Pass Theorem. In the last section we give the proof of Theorem \ref{thm1}.

\section{Preliminaries and functional setting}
Let us consider the fractional Sobolev space 
$$
H^{s}(\R^{3}, \R)=\{u\in L^{2}(\R^{3}, \R): [u]<\infty\}
$$
where
$$
[u]^{2}=\iint_{\R^{6}} \frac{|u(x)-u(y)|^{2}}{|x-y|^{3+2s}} dxdy.
$$
It is well-known (see \cite{DPV, MBRS}) that the embedding $H^{s}(\R^{3}, \R)\subset L^{q}(\R^{3}, \R)$ is continuous for all $q\in [2, \2)$ and locally compact for all $q\in [1, \2)$.

Let $L^{2}(\R^{3}, \C)$ be the space of complex-valued functions such that $\int_{\R^{3}}|u|^{2}\, dx<\infty$ endowed with the inner product 
$\langle u, v\rangle_{L^{2}}=\Re\int_{\R^{3}} u\bar{v}\, dx$, where the bar denotes complex conjugation.

Let us denote by
$$
[u]^{2}_{A}:=\frac{c_{3,s}}{2}\iint_{\R^{6}} \frac{|u(x)-e^{\imath (x-y)\cdot A(\frac{x+y}{2})} u(y)|^{2}}{|x-y|^{3+2s}} \, dxdy,
$$
and we define
$$
D_A^s(\R^3,\C)
:=
\left\{
u\in L^{2_s^*}(\R^3,\C) : [u]^{2}_{A}<\infty
\right\}.
$$
In order to study our problem, we introduce the Hilbert space
$$
H^{s}_{\e}:=
\left\{
u\in D_{A_{\e}}^s(\R^3,\C): \int_{\R^{3}} V(\e x) |u|^{2}\, dx <\infty
\right\}
$$ 
endowed with the scalar product
\begin{align*}
\langle u , v \rangle_{\e}&=
\Re	\int_{\R^{3}} V(\e x) u \bar{v} dx\\
&
+ \frac{c_{3,s}}{2}\Re\iint_{\R^{6}} \frac{(u(x)-e^{\imath(x-y)\cdot A_{\e}(\frac{x+y}{2})} u(y))\overline{(v(x)-e^{\imath(x-y)\cdot A_{\e}(\frac{x+y}{2})}v(y))}}{|x-y|^{3+2s}} dx dy
\end{align*}
and we set
$$
\|u\|_{\e}:=\sqrt{\langle u , u \rangle_{\e}}.
$$
The space $\h$ satisfies the following fundamental properties; see \cite{AD, DS} for more details.
\begin{lem}\cite{AD, DS}
The space $\h$ is complete and $C_c^\infty(\R^3,\C)$ is dense in $\h$. 
\end{lem}

\begin{thm}\label{Sembedding}\cite{DS}
	The space $H^{s}_{\e}$ is continuously embedded in $L^{r}(\R^{3}, \C)$ for $r\in [2, 2^{*}_{s}]$, and compactly embedded in $L_{\rm loc}^{r}(\R^{3}, \C)$ for $r\in [1, 2^{*}_{s})$.\\
%Moreover, if $V_\infty=\infty$, then, for any bounded sequence $(u_{n})$  in $\h$, we have that, up to a subsequence, $(|u_{n}|)$ is strongly convergent in $L^{r}(\R^{3}, \R)$ for $r\in [2, 2^{*}_{s})$.
\end{thm}

\begin{lem}\label{DI}\cite{DS}
If $u\in H^{s}_{A}(\R^{3}, \C)$ then $|u|\in H^{s}(\R^{3}, \R)$ and we have
$$
[|u|]\leq [u]_{A}.
$$
\end{lem}

\begin{lem}\label{aux}\cite{AD}
If $u\in H^{s}(\R^{3}, \R)$ and $u$ has compact support, then $w=e^{\imath A(0)\cdot x} u \in \h$.
\end{lem}
We also recall the following vanishing lemma \cite{FQT} which will be useful for our study:
\begin{lem}\label{Lions}\cite{FQT}
	Let $q\in [2, \2)$. If $(u_{n})$ is a bounded sequence in $H^{s}(\R^{3}, \R)$ and if 
	\begin{equation*}
	\lim_{n\rightarrow \infty} \sup_{y\in \R^{3}} \int_{B_{R}(y)} |u_{n}|^{q} dx=0
	\end{equation*}
	for some $R>0$, then $u_{n}\rightarrow 0$ in $L^{r}(\R^{3}, \R)$ for all $r\in (2, 2^{*}_{s})$.
\end{lem}

\noindent
Now, let $s, t\in (0, 1)$ such that $4s+2t\geq 3$. Since $H^{s}(\R^{3}, \R)\subset L^{q}(\R^{3}, \R)$ for all $q\in [2, \2)$, we can deduce that 
\begin{equation}\label{ruiz}
H^{s}(\R^{3}, \R)\subset L^{\frac{12}{3+2t}}(\R^{3}, \R).
\end{equation}
For any $u\in \h$, we know that $|u|\in H^{s}(\R^{3}, \R)$ in view of Lemma \ref{DI}, and then we  consider the linear functional $\mathcal{L}_{|u|}: D^{t, 2}(\R^{3}, \R)\rightarrow \R$  given by
$$
\mathcal{L}_{|u|}(v)=\int_{\R^{3}} |u|^{2} v\, dx. 
$$
By using H\"older inequality and \eqref{ruiz} we can see that
\begin{equation}
|\mathcal{L}_{|u|}(v)|\leq \left(\int_{\R^{3}} |u|^{\frac{12}{3+2t}}dx\right)^{\frac{3+2t}{6}}  \left(\int_{\R^{3}} |v|^{2^{*}_{t}}dx\right)^{\frac{1}{2^{*}_{t}}}\leq C\|u\|^{2}_{D^{s,2}}\|v\|_{D^{t,2}},
\end{equation}
where
$$
\|v\|^{2}_{D^{t,2}}=\iint_{\R^{6}} \frac{|v(x)-v(y)|^{2}}{|x-y|^{3+2t}}dxdy,
$$
and this shows that $\mathcal{L}_{|u|}$ is well defined and continuous.
By using the Lax-Milgram Theorem, there exists a unique $\phi_{|u|}^{t}\in D^{t, 2}(\R^{3}, \R)$ such that 
\begin{equation}\label{FPE}
(-\Delta)^{t} \phi_{|u|}^{t}= |u|^{2} \mbox{ in } \R^{3}.
\end{equation}
Then we have the following $t$-Riesz formula
\begin{equation}\label{RF}
\phi_{|u|}^{t}(x)=c_{t}\int_{\R^{3}} \frac{|u(y)|^{2}}{|x-y|^{3-2t}} \, dy  \quad (x\in \R^{3}), \quad c_{t}=\pi^{-\frac{3}{2}}2^{-2t}\frac{\Gamma(3-2t)}{\Gamma(t)}.
\end{equation}
In the sequel, we will omit the constant $c_{t}$ in order to lighten the notation. Finally, we prove some properties on the convolution term.
\begin{lem}\label{poisson}
Let us assume that $4s+2t\geq 3$ and $u\in \h$.
%and we define $\Phi(u)=\phi_{|u|}^{t}\in D^{t,2}(\R^{3},\R)$, where 
%$$
%\phi_{|u|}^{t}(x)=c_{t}\int_{\R^{3}} \frac{|u(y)|^{2}}{|x-y|^{3-2t}}dy \, \mbox{ with } c_{t}=\frac{\Gamma(\frac{3}{2}-2t)}{\pi^{3/2}2^{2t}\Gamma(t)}.
%$$ 
Then we have:
\begin{enumerate}
\item 
$\phi_{|u|}^{t}: H^{s}(\R^{3},\R)\rightarrow D^{t,2}(\R^{3}, \R)$ is continuous and maps bounded sets into bounded sets,
\item if $u_{n}\rightharpoonup u$ in $\h$ then $\phi_{|u_{n}|}^{t}\rightharpoonup \phi_{|u|}^{t}$ in $D^{t,2}(\R^{3},\R)$,
\item $\phi^{t}_{|ru|}=r^{2}\phi^{t}_{|u|}$ for all $r\in \R$ and $\phi^{t}_{|u(\cdot+y)|}(x)=\phi^{t}_{|u|}(x+y)$,
\item $\phi_{|u|}^{t}\geq 0$ for all $u\in \h$, and we have
$$
\|\phi_{|u|}^{t}\|_{D^{t,2}}\leq C\|u\|_{L^{\frac{12}{3+2t}}(\R^{3})}^{2}\leq C\|u\|^{2}_{\e} \, \mbox{ and } \int_{\R^{3}} \phi_{|u|}^{t}|u|^{2}dx\leq C\|u\|^{4}_{L^{\frac{12}{3+2t}}(\R^{3})}\leq C\|u\|_{\e}^{4}.
$$
\end{enumerate} 
\end{lem}
\begin{proof}
$(1)$ Since $\phi_{|u|}^{t}\in D^{t,2}(\R^{3}, \R)$ satisfies \eqref{FPE}, that is
$$
\int_{\R^{3}}(-\Delta)^{\frac{t}{2}}\phi_{|u|}^{t} (-\Delta)^{\frac{t}{2}}v \,dx=\int_{\R^{3}} |u|^{2}v \,dx
$$
for all $v\in D^{t,2}(\R^{3}, \R)$, we can see that $\mathcal{L}_{|u|}$ is such that $\|\mathcal{L}_{|u|}\|_{\mathcal{L}(D^{t,2}, \R)}=\|\phi_{|u|}^{t}\|_{D^{t,2}}$ for all $u\in \h$. Hence, in order to prove the continuity of $\phi_{|u|}^{t}$, it is enough to show that the map $u\mapsto \mathcal{L}_{|u|}$ is continuous. Let $u_{n}\rightarrow u$ in $\h$. By using Lemma \ref{DI} and Theorem \ref{Sembedding} we deduce that $|u_{n}|\rightarrow |u|$ in $L^{\frac{12}{3+2t}}(\R^{3})$. Hence, for all $v\in D^{t,2}(\R^{3}, \R)$ we have
\begin{align*}
|\mathcal{L}_{|u_{n}|}(v)-\mathcal{L}_{|u|}(v)|&=\left|\int_{\R^{3}} (|u_{n}|^{2}-|u|^{2})v\, dx\right| \\
&\leq \left(\int_{\R^{3}} ||u_{n}|^{2}-|u|^{2}|^{\frac{6}{3+2t}} \, dx\right)^{\frac{3+2t}{6}} \|v\|_{L^{\frac{6}{3-2t}}(\R^{3})} \\
&\leq C\left[\left(\int_{\R^{3}} ||u_{n}|-|u||^{\frac{12}{3+2t}} \, dx\right)^{\frac{1}{2}} \left(\int_{\R^{3}} ||u_{n}|+|u||^{\frac{12}{3+2t}} \, dx\right)^{\frac{1}{2}}\right]^{\frac{3+2t}{6}} \|v\|_{D^{t,2}} \\
&\leq C  \||u_{n}|-|u|\|_{L^{\frac{12}{3+2t}}(\R^{3})}\|v\|_{D^{t,2}}
\end{align*}
which implies that $\|\phi_{|u_{n}|}^{t}-\phi_{|u|}^{t} \|_{D^{t,2}}=\|\mathcal{L}_{|u_{n}|}-\mathcal{L}_{|u|}\|_{\mathcal{L}(D^{t,2}, \R)}\rightarrow 0$ as $n\rightarrow \infty$.\\
$(2)$ If $u_{n}\rightharpoonup u$ in $\h$, then Lemma \ref{DI} and Theorem \ref{Sembedding} yield $|u_{n}|\rightarrow |u|$ in $L^{q}_{loc}(\R^{3}, \R)$ for all $q\in [1, \2)$. Hence, for all $v\in C^{\infty}_{c}(\R^{3}, \R)$ we get
\begin{align*}
\langle \phi_{|u_{n}|}^{t}-\phi_{|u|}^{t}, v\rangle&=\int_{\R^{3}} (|u_{n}|^{2}-|u|^{2})v\, dx \\
&\leq \left(\int_{supp(v)} ||u_{n}|-|u||^{2}\, dx\right)^{\frac{1}{2}} \left(\int_{\R^{3}} ||u_{n}|+|u||^{2}\, dx\right)^{\frac{1}{2}}  \|v\|_{L^{\infty}(\R^{3})}\\
&\leq C \||u_{n}|-|u|\|_{L^{2}(supp(v))} \|v\|_{L^{\infty}(\R^{3})}\rightarrow 0.
\end{align*}
$(3)$ and $(4)$ are easily obtained by applying Hardy-Littlewood-Sobolev inequality (see Theorem $4.3$ in \cite{LL}), H\"older inequality and Sobolev embedding. 
\end{proof}

\section{The Modified problem}
By using the change of variable $x\mapsto \e x$, we can see that the study of (\ref{P}) is equivalent to consider the following problem
\begin{equation}\label{Pe}
(-\Delta)_{A_{\e}}^{s} u + V_{\e}( x)u +(|x|^{2t-3}*|u|^{2})u=  f(|u|^{2})u  \mbox{ in } \R^{3},
\end{equation}
where $A_{\e}(x)=A(\e x)$ and $V_{\e}(x)=V(\e x)$.

As in \cite{AM, DF}, we  fix $k>\frac{\theta}{\theta-2}$ and $a>0$ such that $\frac{f(a)}{a}=\frac{V_{0}}{k}$, and we introduce the function
$$
\tilde{f}(t):=
\begin{cases}
f(t)& \text{ if $t \leq a$} \\
\frac{V_{0}}{k}    & \text{ if $t >a$}.
\end{cases}
$$ 
Then we define the penalized nonlinearity $g: \R^{3}\times \R\rightarrow \R$ by setting
$$
g(x, t)=\chi_{\Lambda}(x) f(t)+(1-\chi_{\Lambda}(x))\tilde{f}(t),
$$
where $\chi_{\Lambda}$ is the characteristic function on $\Lambda$, and  we set $G(x, t)=\int_{0}^{t} g(x, \tau)\, d\tau$.

From the assumptions $(f_1)$-$(f_4)$ it is standard to check that $g$ verifies the following properties:
\begin{compactenum}[($g_1$)]
\item $\displaystyle{\lim_{t\rightarrow 0} \frac{g(x, t)}{t}=0}$ uniformly in $x\in \R^{3}$;
\item $\lim_{t\rightarrow \infty} \frac{g(x,t)}{t^{\frac{q-2}{2}}}=0$ uniformly in $x\in \R^{3}$;
\item $(i)$ $0< \frac{\theta}{2} G(x, t)\leq g(x, t)t$ for any $x\in \Lambda$ and $t>0$, \\
$(ii)$ $0\leq  G(x, t)\leq g(x, t)t\leq \frac{V(x)}{k}t$ and $0\leq g(x,t)\leq \frac{V(x)}{k}$ for any $x\in \Lambda^{c}$ and $t>0$;
\item $t\mapsto \frac{g(x,t)}{t}$ is increasing for all $x\in \Lambda$ and $t>0$.
\end{compactenum}
Then, we consider the following modified problem 
\begin{equation}\label{MPe}
(-\Delta)^{s}_{A_{\e}} u + V(\e x)u +\phi_{|u|}^{t}u=  g_{\e}(x, |u|^{2})u \mbox{ in } \R^{3}, 
\end{equation}
where $g_{\e}(x, t)=g(\e x, t)$ and $\phi_{|u|}^{t}$ is given by \eqref{RF}.

Let us note that if $u$ is a solution of (\ref{MPe}) such that 
\begin{equation}\label{ue}
|u(x)|\leq a \mbox{ for all } x\in  \Lambda_{\e}^{c},
\end{equation}
where $\Lambda_{\e}:=\{x\in \R^{N}: \e x\in \Lambda\}$, then $u$ is also a solution of the original problem  (\ref{Pe}).

In order to find weak solutions to (\ref{Pe}), we look for critical points of the Euler-Lagrange functional $J_{\e}: \h\rightarrow \R$ defined as
$$
J_{\e}(u)=\frac{1}{2}\int_{\R^{3}}|(-\Delta)_{A_{\e}}^{\frac{s}{2}}u|^{2}+V(\e x) |u|^{2}\, dx+\frac{1}{4}\int_{\R^{3}}\phi_{|u|}^{t}|u|^{2}dx-\frac{1}{2}\int_{\R^{3}} G(\e x, |u|^{2})\, dx.
$$
%When $\e=0$, $\h$ reduces to the well-known fractional space $H^{s}(\R^{N}, \R)$; see \cite{DPV, MBRS} for more details.
We also consider the autonomous problem associated to \eqref{Pe}, that is
\begin{equation}\label{APe}
(-\Delta)^{s} u + V_{0}u +\phi^{t}_{|u|}u=  f(u^{2})u \mbox{ in } \R^{3}, 
\end{equation}
and we denote by $J_{0}: H^{s}(\R^{3}, \R)\rightarrow \R$ the corresponding energy functional
\begin{align*}
J_{0}(u)&=\frac{1}{2}\int_{\R^{3}}|(-\Delta)^{\frac{s}{2}}u|^{2}+V_{0} |u|^{2}\, dx+\frac{1}{4}\int_{\R^{3}}\phi^{t}_{|u|}u^{2}dx-\frac{1}{2}\int_{\R^{3}} F(u^{2})\, dx\\
&=\frac{1}{2}\|u\|^{2}_{0}+\frac{1}{4}\int_{\R^{3}}\phi^{t}_{|u|}u^{2}dx-\frac{1}{2}\int_{\R^{3}} F(u^{2})\, dx
\end{align*}
where we used the notation $\|\cdot\|_{0}$ to indicate the $H^{s}(\R^{3}, \R)$-norm (equivalent to the standard one).

In what follows, we show that $J_{\e}$ verifies the assumptions of the Mountain Pass Theorem \cite{AR}. 
\begin{lem}\label{MPG}
The functional $J_{\e}$ possesses a Mountain Pass geometry:
\begin{compactenum}[$(i)$]
\item $J_{\e}(0)=0$;
\item there exists $\alpha, \rho>0$ such that $J_{\e}(u)\geq \alpha$ for any $u\in \h$ such that $\|u\|_{\e}=\rho$;
\item there exists $e\in \h$ with $\|e\|_{\e}>\rho$ such that $J_{\e}(e)<0$.
\end{compactenum}
\end{lem}
\begin{proof}
The condition $(i)$ is obvious. By using $(g_1)$ and $(g_2)$, and Theorem \ref{Sembedding} we can see that for any $\delta>0$ there exists $C_{\delta}>0$ such that
$$
J_{\e}(u)\geq \frac{1}{2}\|u\|^{2}_{\e}-\delta C\|u\|^{4}_{\e}-C_{\delta} \|u\|^{q}_{\e}.
$$
Choosing $\delta>0$ sufficiently small, we can see that $(ii)$ holds.
Regarding $(iii)$, we can note that in view of $(g_3)$, we have for any $u\in \h\setminus\{0\}$ with $supp(u)\subset \Lambda_{\e}$ and $T>1$
\begin{align*}
J_{\e}(Tu)&\leq \frac{T^{2}}{2} \|u\|^{2}_{\e}+\frac{T^{4}}{4}\int_{\R^{3}}\phi_{|u|}^{t}|u|^{2}dx-\frac{1}{2}\int_{\Lambda_{\e}} G(\e x, T^{2}|u|^{2})\, dx \\
&\leq \frac{T^{4}}{2} \left(\|u\|^{2}_{\e}+\int_{\R^{3}}\phi_{|u|}^{t}|u|^{2}dx\right)-CT^{\theta} \int_{\Lambda_{\e}} |u|^{\theta}\, dx+C
\end{align*}
which implies that $J_{\e}(Tu)\rightarrow -\infty$ as $T\rightarrow \infty$ being $\theta>4$.
\end{proof}

\begin{lem}\label{PSc}
Let $c\in \R$. Then $J_{\e}$ satisfies the Palais-Smale condition at the level $c$.
\end{lem}
\begin{proof}
Let $(u_{n})\subset \h$ be a $(PS)$-sequence. Then $(u_{n})$ is bounded in $\h$. Indeed, using $(g_3)$ we have
\begin{align*}
c+o_{n}(1)\|u_{n}\|_{\e}&= J_{\e}(u_{n})-\frac{1}{\theta}\langle J'_{\e}(u_{n}), u_{n}\rangle \\
&= \left(\frac{1}{2}-\frac{1}{\theta}\right)\|u_{n}\|^{2}_{\e}+\left( \frac{1}{4}-\frac{1}{\theta}\right)\int_{\R^{3}} \phi_{|u_{n}|}^{t}|u_{n}|^{2}dx\\
&+\frac{1}{\theta}\int_{\R^{3}} g_{\e}(x, |u_{n}|^{2})|u_{n}|^{2}\, dx-\frac{1}{2} \int_{\R^{3}} G_{\e}(x, |u_{n}|^{2})\, dx\\
&\geq \left(\frac{1}{2}-\frac{1}{\theta}\right)\|u_{n}\|^{2}_{\e}+\frac{1}{\theta}\int_{\Lambda_{\e}} \left(g_{\e}(x, |u_{n}|^{2})|u_{n}|^{2}-\frac{\theta}{2}G(\e x, |u_{n}|^{2}) \right)dx  \\
&+\frac{1}{\theta}\int_{\Lambda^{c}_{\e}} g_{\e}(x, |u_{n}|^{2})|u_{n}|^{2}\, dx-\frac{1}{2} \int_{\R^{3}}\frac{V(\e x)}{k}|u_{n}|^{2} dx \\
&\geq \frac{1}{2}\left(\frac{\theta-2}{\theta}-\frac{1}{k}\right)\|u_{n}\|^{2}_{\e},
\end{align*}
and recalling that $k>\frac{\theta}{\theta-2}$ we get the thesis.
Now, we show that for any $\xi>0$ there exists $R=R_{\xi}>0$ such that $\Lambda_{\e}\subset B_{R}$ and
\begin{equation}\label{T}
\limsup_{n\rightarrow \infty}\int_{B_{R}^{c}} |(-\Delta)^{s}_{A_{\e}}u_{n}|^{2}+V_{\e}(x)|u_{n}|^{2}\, dx\leq \xi.
\end{equation}
Assume for the moment that that the above claim holds, and we show how this information can be used.
By using $u_{n}\rightharpoonup u$ in $\h$, Theorem \ref{Sembedding} and $(g_1)$-$(g_2)$, it is easy to see that
\begin{align}\label{Poissw0}
(u_{n}, \psi)_{\e}\rightarrow (u, \psi)_{\e} \mbox{ and } \Re\left(\int_{\R^{3}} g(\e x, |u_{n}|^{2}) u_{n}\bar{\psi} dx\right)\rightarrow  \Re\left(\int_{\R^{3}} g(\e x, |u|^{2}) u\bar{\psi} dx\right). 
\end{align}
%Moreover $u_{n}\rightarrow u$ in $L^{q}(K, \C)$ for any $q\in [2, 2^{*}_{s})$ and $K\subset \R^{3}$ compact, and using   \eqref{T} and Theorem \ref{Sembedding} we can see that for all $\xi>0$ there exists $R=R_{\xi}>0$ such that for any $n$ large enough
Moreover, by using \eqref{T} and Theorem \ref{Sembedding} we can see that for all $\xi>0$ there exists $R=R_{\xi}>0$ such that for any $n$ large enough
\begin{align*}
\|u_{n}-u\|_{L^{q}(\R^{3})}&=\|u_{n}-u\|_{L^{q}(B_{R})}+\|u_{n}-u\|_{L^{q}(B_{R}^{c})}\\
&\leq \|u_{n}-u\|_{L^{q}(B_{R})}+(\|u_{n}\|_{L^{q}(B_{R}^{c})}+\|u\|_{L^{q}(B^{c}_{R})}) \\
&\leq \xi+2C\xi,
\end{align*}
where $q\in [2, \2)$, which gives  
\begin{equation}\label{CSSq}
u_{n}\rightarrow u \mbox{ in } L^{q}(\R^{3}, \C)\quad \forall q\in [2, \2).
\end{equation}
Since $||u_{n}|-|u||\leq |u_{n}-u|$ and $\frac{12}{3+2t}\in (2, 2^{*}_{s})$, we also have $|u_{n}|\rightarrow |u|$ in $L^{\frac{12}{3+2t}}(\R^{3},\R)$. \\
Then, recalling that $\phi_{|u|}: L^{\frac{12}{3+27}}(\R^{3}, \R)\rightarrow D^{t,2}(\R^{3}, \R)$ is continuous (see Lemma \ref{poisson}) we can deduce that
\begin{align}\begin{split}\label{Poissw1}
\phi_{|u_{n}|}^{t}\rightarrow \phi_{|u|}^{t} \mbox{ in } D^{t,2}(\R^{3}, \R).
\end{split}\end{align}
Putting together \eqref{CSSq}, \eqref{Poissw1}, H\"older inequality and Theorem \ref{Sembedding} we obtain
\begin{align}\label{Poissw2}
\Re\left(\int_{\R^{3}} (\phi_{|u_{n}|}^{t}u_{n}-\phi_{|u|}^{t}u)\bar{\psi}dx  \right)&=\Re\left(\int_{\R^{3}} \phi_{|u_{n}|}^{t}(u_{n}-u)\bar{\psi}+\int_{\R^{3}} (\phi_{|u_{n}|}^{t}-\phi_{|u|}^{t})u\bar{\psi}dx  \right) \nonumber\\
&\leq \|\phi_{|u_{n}|}^{t}\|_{L^{\frac{6}{3+2t}}(\R^{3})} \|u_{n}-u\|_{L^{\frac{12}{3+2t}}(\R^{3})}\|\psi\|_{L^{\frac{12}{3+2t}}(\R^{3})}\nonumber \\
&+\|\phi_{|u_{n}|}^{t}-\phi_{|u|}^{t}\|_{\frac{6}{3+2t}} \|u\|_{L^{\frac{12}{3+2t}}(\R^{3})}\|\psi\|_{\frac{12}{3+2t}} \nonumber\\
&\leq C \|u_{n}-u\|_{L^{\frac{12}{3+2t}}(\R^{3})}+C\|\phi_{|u_{n}|}^{t}-\phi_{|u|}^{t}\|_{D^{t,2}}\rightarrow 0.
\end{align}
Now, we show that
\begin{equation}\label{Poiss3}
\int_{\R^{3}} \phi_{|u_{n}|}^{t}|u_{n}|^{2}dx\rightarrow \int_{\R^{3}} \phi_{|u|}^{t}|u|^{2}dx.
\end{equation}
Let us start proving that
\begin{equation*}
|\mathbb{D}(u_{n})-\mathbb{D}(u)|\leq \sqrt{\mathbb{D}(||u_{n}|^{2}-|u|^{2}|^{1/2})} \sqrt{\mathbb{D}(||u_{n}|^{2}+|u|^{2}|^{1/2})},
\end{equation*}
where
$$
\mathbb{D}(u)=\iint_{\R^{6}} |x-y|^{-(3-2t)}|u(x)|^{2}|u(y)|^{2}dxdy.
$$
Indeed, taking into account $|x|^{-(3-2t)}$ is even and Theorem $9.8$ in \cite{LL} (see Remark after Theorem $9.8$ and recall that $-3<-(3-2t)<0$ ) we have
\begin{align*}
|\mathbb{D}(u_{n})-\mathbb{D}(u)|&=\left|\iint_{\R^{6}} |x-y|^{-(3-2t)}|u_{n}(x)|^{2}|u_{n}(y)|^{2}dxdy-\iint_{\R^{6}} |x-y|^{-(3-2t)}|u(x)|^{2}|u(y)|^{2}dxdy\right| \\
&=\Bigl|\iint_{\R^{6}} |x-y|^{-(3-2t)}|u_{n}(x)|^{2}|u_{n}(y)|^{2}dxdy+\iint_{\R^{6}} |x-y|^{-(3-2t)}|u_{n}(x)|^{2}|u(y)|^{2}dxdy \\
&-\iint_{\R^{6}} |x-y|^{-(3-2t)}|u(x)|^{2}|u_{n}(y)|^{2}dxdy-\iint_{\R^{6}} |x-y|^{-(3-2t)}|u(x)|^{2}|u(y)|^{2}dxdy\Bigr| \\
&=\left|\iint_{\R^{6}} |x-y|^{-(3-2t)}(|u_{n}(x)|^{2}-|u(x)|^{2}|)(|u_{n}(y)|^{2}+|u(y)|^{2}) dxdy  \right| \\
&\leq \iint_{\R^{6}} |x-y|^{-(3-2t)}||u_{n}(x)|^{2}-|u(x)|^{2}|| ||u_{n}(y)|^{2}+|u(y)|^{2}| dxdy  \\
&\leq C\sqrt{\mathbb{D}(||u_{n}|^{2}-|u|^{2}|^{1/2})} \sqrt{\mathbb{D}(||u_{n}|^{2}+|u|^{2}|^{1/2})}.
\end{align*}
Thus, by using Hardy-Littlewood-Sobolev inequality (see Theorem $4.3$ in \cite{LL}), H\"older inequality, the boundedness of $(|u_{n}|)$ in $H^{s}(\R^{3}, \R)$ and $|u_{n}|\rightarrow |u|$ in $L^{\frac{12}{3+2t}}(\R^{3},\R)$ we can see that
\begin{align*}
|\mathbb{D}(u_{n})-\mathbb{D}(u)|^{2}&\leq C\|||u_{n}|^{2}-|u|^{2}||^{1/2}\|^{4}_{L^{\frac{12}{3+2t}}(\R^{3})} \|||u_{n}|^{2}+|u|^{2}||^{1/2}\|^{4}_{L^{\frac{12}{3+2t}}(\R^{3})} \nonumber \\
&\leq C\||u_{n}|-|u|\|^{2}_{L^{\frac{12}{3+2t}}(\R^{3})}\rightarrow 0. 
\end{align*}

Therefore, by using $\langle J'_{\e}(u_{n}), \psi\rangle =o_{n}(1)$ for all $\psi\in C^{\infty}_{c}(\R^{3}, \C)$, and taking into account \eqref{Poissw0} and \eqref{Poissw2}, we can check that $J'_{\e}(u)=0$. In particular
\begin{equation}\label{Poiss0}
\|u\|^{2}_{\e}+\int_{\R^{3}} \phi_{|u|}^{t}|u|^{2}dx=\int_{\R^{3}} g_{\e}(x, |u|^{2})|u|^{2}\, dx.
\end{equation}
Now, we know that $\langle J'_{\e}(u_{n}), u_{n}\rangle =o_{n}(1)$ is equivalent to
\begin{equation}\label{Poiss1}
\|u_{n}\|^{2}_{\e}+\int_{\R^{3}} \phi_{|u_{n}|}^{t}|u_{n}|^{2}dx=\int_{\R^{3}} g_{\e}(x, |u_{n}|^{2})|u_{n}|^{2}\, dx+o_{n}(1).
\end{equation}
From the growth assumptions $(g_1)$-$(g_2)$ and using \eqref{T}, we can see that
\begin{equation}\label{Gng}
\int_{\R^{3}} g_{\e}(x, |u_{n}|^{2})|u_{n}|^{2}\, dx\rightarrow \int_{\R^{3}} g_{\e}(x, |u|^{2})|u|^{2}\, dx.
\end{equation}
Then, taking into account  \eqref{Poiss3}, \eqref{Poiss0},  \eqref{Poiss1} and  \eqref{Gng} we can infer that
$$
\lim_{n\rightarrow \infty}\|u_{n}\|^{2}_{\e}=\|u\|^{2}_{\e}.
$$
It remains to prove that \eqref{T} holds. Let $\eta_{R}\in C^{\infty}(\R^{3}, \R)$ be such that $0\leq \eta_{R}\leq 1$, $\eta_{R}=0$ in $B_{\frac{R}{2}}$, $\eta_{R}=1$ in $B_{R}^{c}$ and $|\nabla \eta_{R}|\leq \frac{C}{R}$ for some $C>0$ independent of $R$.
Since $(u_{n}\eta_{R})$ is bounded, we can see that $\langle J'_{\e}(u_{n}), u_{n}\eta_{R}\rangle =o_{n}(1)$, that is
\begin{align*}
&\Re\left(\iint_{\R^{6}} \frac{(u_{n}(x)-u_{n}(y)e^{\imath A_{\e}(\frac{x+y}{2})\cdot (x-y)})\overline{(u_{n}(x)\eta_{R}(x)-u_{n}(y)\eta_{R}(y)e^{\imath A_{\e}(\frac{x+y}{2})\cdot (x-y)})}}{|x-y|^{3+2s}}\, dx dy \right)\\
&+\int_{\R^{3}} \phi_{|u_{n}|}^{t} |u_{n}|^{2}\eta_{R} dx+\int_{\R^{3}} V_{\e}(x)\eta_{R} |u_{n}|^{2}\, dx=\int_{\R^{3}} g_{\e}(x, |u_{n}|^{2})|u_{n}|^{2}\eta_{R}\, dx+o_{n}(1).
\end{align*}
From
\begin{align*}
&\Re\left(\iint_{\R^{6}} \frac{(u_{n}(x)-u_{n}(y)e^{\imath A_{\e}(\frac{x+y}{2})\cdot (x-y)})\overline{(u_{n}(x)\eta_{R}(x)-u_{n}(y)\eta_{R}(y)e^{\imath A_{\e}(\frac{x+y}{2})\cdot (x-y)})}}{|x-y|^{3+2s}}\, dx dy \right)\\
&=\Re\left(\iint_{\R^{6}} \overline{u_{n}(y)}e^{-\imath A_{\e}(\frac{x+y}{2})\cdot (x-y)}\frac{(u_{n}(x)-u_{n}(y)e^{\imath A_{\e}(\frac{x+y}{2})\cdot (x-y)})(\eta_{R}(x)-\eta_{R}(y))}{|x-y|^{3+2s}}  \,dx dy\right)\\
&+\iint_{\R^{6}} \eta_{R}(x)\frac{|u_{n}(x)-u_{n}(y)e^{\imath A_{\e}(\frac{x+y}{2})\cdot (x-y)}|^{2}}{|x-y|^{3+2s}}\, dx dy,
\end{align*}
and using $(g_3)$-$(ii)$ and Lemma \ref{poisson}-$(4)$, it follows that
\begin{align}\label{PS1}
&\iint_{\R^{6}} \eta_{R}(x)\frac{|u_{n}(x)-u_{n}(y)e^{\imath A_{\e}(\frac{x+y}{2})\cdot (x-y)}|^{2}}{|x-y|^{3+2s}}\, dx dy+\int_{\R^{3}} V_{\e}(x)\eta_{R} |u_{n}|^{2}\, dx\nonumber\\
&\leq -\Re\left(\iint_{\R^{6}} \overline{u_{n}(y)}e^{-\imath A_{\e}(\frac{x+y}{2})\cdot (x-y)}\frac{(u_{n}(x)-u_{n}(y)e^{\imath A_{\e}(\frac{x+y}{2})\cdot (x-y)})(\eta_{R}(x)-\eta_{R}(y))}{|x-y|^{3+2s}}  \,dx dy\right) \nonumber\\
&+\frac{1}{k}\int_{\R^{3}} V_{\e}(x) \eta_{R} |u_{n}|^{2}\, dx+o_{n}(1).
\end{align}
Now, by using H\"older inequality and the boundedness of $(u_{n})$ in $\h$ we get
\begin{align}\label{PS2}
&\left|\Re\left(\iint_{\R^{6}} \overline{u_{n}(y)}e^{-\imath A_{\e}(\frac{x+y}{2})\cdot (x-y)}\frac{(u_{n}(x)-u_{n}(y)e^{\imath A_{\e}(\frac{x+y}{2})\cdot (x-y)})(\eta_{R}(x)-\eta_{R}(y))}{|x-y|^{3+2s}}  \,dx dy\right)\right| \nonumber\\
&\leq \left(\iint_{\R^{6}} \frac{|u_{n}(x)-u_{n}(y)e^{\imath A_{\e}(\frac{x+y}{2})\cdot (x-y)}|^{2}}{|x-y|^{3+2s}}\,dxdy  \right)^{\frac{1}{2}} \left(\iint_{\R^{6}} |\overline{u_{n}(y)}|^{2}\frac{|\eta_{R}(x)-\eta_{R}(y)|^{2}}{|x-y|^{3+2s}} \, dxdy\right)^{\frac{1}{2}} \nonumber\\
&\leq C \left(\iint_{\R^{6}} |u_{n}(y)|^{2}\frac{|\eta_{R}(x)-\eta_{R}(y)|^{2}}{|x-y|^{3+2s}} \, dxdy\right)^{\frac{1}{2}}.
\end{align}
In what follows, we show that
%Arguing as in \cite{A6} we can prove that
\begin{equation}\label{PS3}
\limsup_{R\rightarrow \infty}\limsup_{n\rightarrow \infty} \iint_{\R^{6}} |u_{n}(y)|^{2}\frac{|\eta_{R}(x)-\eta_{R}(y)|^{2}}{|x-y|^{3+2s}} \, dxdy=0.
\end{equation}
Let us note that  
$$
\R^{6}=((\R^{3}\setminus B_{2R})\times (\R^{3}\setminus B_{2R})) \cup ((\R^{3}\setminus B_{2R})\times B_{2R})\cup (B_{2R}\times \R^{3})=: X^{1}_{R}\cup X^{2}_{R} \cup X^{3}_{R}.
$$
Therefore
\begin{align}\label{Pa1}
&\iint_{\R^{6}}\frac{|\eta_{R}(x)-\eta_{R}(y)|^{2}}{|x-y|^{3+2s}} |u_{n}(x)|^{2} dx dy =\iint_{X^{1}_{R}}\frac{|\eta_{R}(x)-\eta_{R}(y)|^{2}}{|x-y|^{3+2s}} |u_{n}(x)|^{2} dx dy \nonumber \\
&+\iint_{X^{2}_{R}}\frac{|\eta_{R}(x)-\eta_{R}(y)|^{2}}{|x-y|^{3+2s}} |u_{n}(x)|^{2} dx dy+
\iint_{X^{3}_{R}}\frac{|\eta_{R}(x)-\eta_{R}(y)|^{2}}{|x-y|^{3+2s}} |u_{n}(x)|^{2} dx dy.
\end{align}
%Now, we estimate each integral in (\ref{Pa1}).
Since $\eta_{R}=1$ in $\R^{3}\setminus B_{2R}$, we can see that
\begin{align}\label{Pa2}
\iint_{X^{1}_{R}}\frac{|u_{n}(x)|^{2}|\eta_{R}(x)-\eta_{R}(y)|^{2}}{|x-y|^{3+2s}} dx dy=0.
\end{align}
Now, fix $k>4$, and we observe that
\begin{equation*}
X^{2}_{R}=(\R^{3} \setminus B_{2R})\times B_{2R} \subset ((\R^{3}\setminus B_{kR})\times B_{2R})\cup ((B_{kR}\setminus B_{2R})\times B_{2R}) 
\end{equation*}
If $(x, y) \in (\R^{3}\setminus B_{kR})\times B_{2R}$, then
\begin{equation*}
|x-y|\geq |x|-|y|\geq |x|-2R>\frac{|x|}{2}. 
\end{equation*}
Therefore, using $0\leq \eta_{R}\leq 1$, $|\nabla \eta_{R}|\leq \frac{C}{R}$ and applying H\"older inequality we obtain
\begin{align}\label{Pa3}
&\iint_{X^{2}_{R}}\frac{|u_{n}(x)|^{2}|\eta_{R}(x)-\eta_{R}(y)|^{2}}{|x-y|^{3+2s}} dx dy \nonumber \\
&=\int_{\R^{3}\setminus B_{kR}} \int_{B_{2R}} \frac{|u_{n}(x)|^{2}|\eta_{R}(x)-\eta_{R}(y)|^{2}}{|x-y|^{3+2s}} dx dy + \int_{B_{kR}\setminus B_{2R}} \int_{B_{2R}} \frac{|u_{n}(x)|^{2}|\eta_{R}(x)-\eta_{R}(y)|^{2}}{|x-y|^{3+2s}} dx dy \nonumber \\
&\leq 2^{2+3+2s} \int_{\R^{3}\setminus B_{kR}} \int_{B_{2R}} \frac{|u_{n}(x)|^{2}}{|x|^{3+2s}}\, dxdy+ \frac{C}{R^{2}} \int_{B_{kR}\setminus B_{2R}} \int_{B_{2R}} \frac{|u_{n}(x)|^{2}}{|x-y|^{3+2(s-1)}}\, dxdy \nonumber \\
&\leq CR^{3} \int_{\R^{3}\setminus B_{kR}} \frac{|u_{n}(x)|^{2}}{|x|^{3+2s}}\, dx + \frac{C}{R^{2}} (kR)^{2(1-s)} \int_{B_{kR}\setminus B_{2R}} |u_{n}(x)|^{2} dx \nonumber \\
&\leq CR^{3} \left( \int_{\R^{3}\setminus B_{kR}} |u_{n}(x)|^{2^{*}_{s}} dx \right)^{\frac{2}{2^{*}_{s}}} \left(\int_{\R^{3}\setminus B_{kR}}\frac{1}{|x|^{\frac{3^{2}}{2s} +3}}\, dx \right)^{\frac{2s}{3}} + \frac{C k^{2(1-s)}}{R^{2s}} \int_{B_{kR}\setminus B_{2R}} |u_{n}(x)|^{2} dx \nonumber \\
&\leq \frac{C}{k^{3}} \left( \int_{\R^{3}\setminus B_{kR}} |u_{n}(x)|^{2^{*}_{s}} dx \right)^{\frac{2}{2^{*}_{s}}} + \frac{C k^{2(1-s)}}{R^{2s}} \int_{B_{kR}\setminus B_{2R}} |u_{n}(x)|^{2} dx \nonumber \\
&\leq \frac{C}{k^{3}}+ \frac{C k^{2(1-s)}}{R^{2s}} \int_{B_{kR}\setminus B_{2R}} |u_{n}(x)|^{2} dx.
\end{align}
Take $\e\in (0,1)$, and we obtain
\begin{align}\label{Ter1}
&\iint_{X^{3}_{R}} \frac{|u_{n}(x)|^{2} |\eta_{R}(x)- \eta_{R}(y)|^{2}}{|x-y|^{3+2s}}\, dxdy \nonumber\\
&\leq \int_{B_{2R}\setminus B_{\varepsilon R}} \int_{\R^{3}} \frac{|u_{n}(x)|^{2} |\eta_{R}(x)- \eta_{R}(y)|^{2}}{|x-y|^{3+2s}}\, dxdy + \int_{B_{\varepsilon R}} \int_{\R^{3}} \frac{|u_{n}(x)|^{2} |\eta_{R}(x)- \eta_{R}(y)|^{2}}{|x-y|^{3+2s}}\, dxdy. 
\end{align}
%Let us estimate the first integral in \eqref{Ter1}. Then, 
Since
\begin{align*}
\int_{B_{2R}\setminus B_{\varepsilon R}} \int_{\R^{3} \cap \{y: |x-y|<R\}} \frac{|u_{n}(x)|^{2} |\eta_{R}(x)- \eta_{R}(y)|^{2}}{|x-y|^{3+2s}}\, dxdy \leq \frac{C}{R^{2s}} \int_{B_{2R}\setminus B_{\varepsilon R}} |u_{n}(x)|^{2} dx
\end{align*}
and 
\begin{align*}
\int_{B_{2R}\setminus B_{\varepsilon R}} \int_{\R^{3} \cap \{y: |x-y|\geq R\}} \frac{|u_{n}(x)|^{2} |\eta_{R}(x)- \eta_{R}(y)|^{2}}{|x-y|^{3+2s}}\, dxdy \leq \frac{C}{R^{2s}} \int_{B_{2R}\setminus B_{\varepsilon R}} |u_{n}(x)|^{2} dx
\end{align*}
we can see that
%from which we have
\begin{align}\label{Ter2}
\int_{B_{2R}\setminus B_{\varepsilon R}} \int_{\R^{3}} \frac{|u_{n}(x)|^{2} |\eta_{R}(x)- \eta_{R}(y)|^{2}}{|x-y|^{3+2s}}\, dxdy \leq \frac{C}{R^{2s}} \int_{B_{2R}\setminus B_{\varepsilon R}} |u_{n}(x)|^{2} dx. 
\end{align}
On the other hand, from the definition of $\eta_{R}$, $\e\in (0,1)$, and $\eta_{R}\leq 1$ we obtain
%Now, by using the definition of $\eta_{R}$, $\e\in (0,1)$, and $\eta_{R}\leq 1$, we have 
\begin{align}\label{Ter3}
\int_{B_{\varepsilon R}} \int_{\R^{3}} \frac{|u_{n}(x)|^{2} |\eta_{R}(x)- \eta_{R}(y)|^{2}}{|x-y|^{3+2s}}\, dxdy &= \int_{B_{\varepsilon R}} \int_{\R^{3}\setminus B_{R}} \frac{|u_{n}(x)|^{2} |\eta_{R}(x)- \eta_{R}(y)|^{2}}{|x-y|^{3+2s}}\, dxdy\nonumber \\
&\leq 4 \int_{B_{\varepsilon R}} \int_{\R^{3}\setminus B_{R}} \frac{|u_{n}(x)|^{2}}{|x-y|^{3+2s}}\, dxdy\nonumber \\
&\leq C \int_{B_{\varepsilon R}} |u_{n}|^{2} dx \int_{(1-\e)R}^{\infty} \frac{1}{r^{1+2s}} dr\nonumber \\
&=\frac{C}{[(1-\e)R]^{2s}} \int_{B_{\varepsilon R}} |u_{n}|^{2} dx
\end{align}
where we use the fact that if $(x, y) \in B_{\varepsilon R}\times (\R^{3} \setminus B_{R})$, then $|x-y|>(1-\e)R$. \\
Then \eqref{Ter1}, \eqref{Ter2} and \eqref{Ter3} yield
%Taking into account \eqref{Ter1}, \eqref{Ter2} and \eqref{Ter3} we deduce 
\begin{align}\label{Pa4}
\iint_{X^{3}_{R}} &\frac{|u_{n}(x)|^{2} |\eta_{R}(x)- \eta_{R}(y)|^{2}}{|x-y|^{3+2s}}\, dxdy \nonumber \\
&\leq \frac{C}{R^{2s}} \int_{B_{2R}\setminus B_{\varepsilon R}} |u_{n}(x)|^{p} dx + \frac{C}{[(1-\e)R]^{2s}} \int_{B_{\varepsilon R}} |u_{n}(x)|^{2} dx. 
\end{align}
In view of \eqref{Pa1}, \eqref{Pa2}, \eqref{Pa3} and \eqref{Pa4} we can infer 
%Putting together \eqref{Pa1}, \eqref{Pa2}, \eqref{Pa3} and \eqref{Pa4}, we can infer 
\begin{align}\label{Pa5}
\iint_{\R^{6}} &\frac{|u_{n}(x)|^{2} |\eta_{R}(x)- \eta_{R}(y)|^{2}}{|x-y|^{3+2s}}\, dxdy \nonumber \\
&\leq \frac{C}{k^{3}} + \frac{Ck^{2(1-s)}}{R^{2s}} \int_{B_{kR}\setminus B_{2R}} |u_{n}(x)|^{2} dx + \frac{C}{R^{2s}} \int_{B_{2R}\setminus B_{\varepsilon R}} |u_{n}(x)|^{2} dx + \frac{C}{[(1-\e)R]^{2s}}\int_{B_{\varepsilon R}} |u_{n}(x)|^{2} dx. 
\end{align}
Since $(|u_{n}|)$ is bounded in $H^{s}(\R^{3}, \R)$, by using Sobolev embedding $H^{s}(\R^{3}, \R)\subset L^{\2}(\R^{3}, \R)$ (see \cite{DPV}), we may assume that $|u_{n}|\rightarrow u$ in $L^{2}_{loc}(\R^{3}, \R)$ for some $u\in H^{s}(\R^{3}, \R)$. Letting the limit as $n\rightarrow \infty$ in \eqref{Pa5} we find
\begin{align*}
&\limsup_{n\rightarrow \infty} \iint_{\R^{6}} \frac{|u_{n}(x)|^{2} |\eta_{R}(x)- \eta_{R}(y)|^{2}}{|x-y|^{3+2s}}\, dxdy\\
&\leq \frac{C}{k^{3}} + \frac{Ck^{2(1-s)}}{R^{2s}} \int_{B_{kR}\setminus B_{2R}} |u(x)|^{2} dx + \frac{C}{R^{2s}} \int_{B_{2R}\setminus B_{\varepsilon R}} |u(x)|^{2} dx + \frac{C}{[(1-\e)R]^{2s}}\int_{B_{\varepsilon R}} |u(x)|^{2} dx \\
&\leq \frac{C}{k^{3}} + Ck^{2} \left( \int_{B_{kR}\setminus B_{2R}} |u(x)|^{2^{*}_{s}} dx\right)^{\frac{2}{2^{*}_{s}}} + C\left(\int_{B_{2R}\setminus B_{\varepsilon R}} |u(x)|^{2^{*}_{s}} dx\right)^{\frac{2}{2^{*}_{s}}} + C\left( \frac{\e}{1-\e}\right)^{2s} \left(\int_{B_{\varepsilon R}} |u(x)|^{2^{*}_{s}} dx\right)^{\frac{2}{2^{*}_{s}}}, 
\end{align*}
where in the last passage we used H\"older inequality. 
Since $u\in L^{2^{*}_{s}}(\R^{3}, \R)$, $k>4$ and $\e \in (0,1)$ we can see that
\begin{align*}
\limsup_{R\rightarrow \infty} \int_{B_{kR}\setminus B_{2R}} |u(x)|^{2^{*}_{s}} dx = \limsup_{R\rightarrow \infty} \int_{B_{2R}\setminus B_{\varepsilon R}} |u(x)|^{2^{*}_{s}} dx = 0.
\end{align*}
Thus, taking $\e= \frac{1}{k}$, we have
\begin{align*}
&\limsup_{R\rightarrow \infty} \limsup_{n\rightarrow \infty} \iint_{\R^{6}} \frac{|u_{n}(x)|^{2} |\eta_{R}(x)- \eta_{R}(y)|^{2}}{|x-y|^{3+2s}}\, dxdy\\
&\leq \lim_{k\rightarrow \infty} \limsup_{R\rightarrow \infty} \Bigl[\, \frac{C}{k^{3}} + Ck^{2} \left( \int_{B_{kR}\setminus B_{2R}} |u(x)|^{2^{*}_{s}} dx\right)^{\frac{2}{2^{*}_{s}}} + C\left(\int_{B_{2R}\setminus B_{\frac{1}{k} R}} |u(x)|^{2^{*}_{s}} dx\right)^{\frac{2}{2^{*}_{s}}} \\
&+ C\left(\frac{1}{k-1}\right)^{2s} \left(\int_{B_{\frac{1}{k} R}} |u(x)|^{2^{*}_{s}} dx\right)^{\frac{2}{2^{*}_{s}}}\, \Bigr]\\
&\leq \lim_{k\rightarrow \infty} \frac{C}{k^{3}} + C\left(\frac{1}{k-1}\right)^{2s} \left(\int_{\R^{3}} |u(x)|^{2^{*}_{s}} dx \right)^{\frac{2}{2^{*}_{s}}}= 0,
\end{align*}
which implies that \eqref{PS3} holds true.  
Putting together \eqref{PS1}, \eqref{PS2} and \eqref{PS3} we can deduce that
$$
\limsup_{R\rightarrow \infty}\limsup_{n\rightarrow \infty} \left(1-\frac{1}{k}\right) \int_{B_{R}^{c}} |(-\Delta)^{\frac{s}{2}}_{A_{\e}}u_{n}|^{2}+V_{\e}(x)|u_{n}|^{2}\, dx=0,
$$
and this completes the proof of \eqref{T}.

\end{proof}

\noindent
In view of Lemma \ref{MPG}, we can define the mountain pass level
$$
c_{\e}=\inf_{\gamma\in \Gamma_{\e}} \max_{t\in [0, 1]} J_{\e}(\gamma(t))
$$
where
$$
\Gamma_{\e}=\{\gamma\in C([0, 1], \h): \gamma(0)=0 \mbox{ and } J_{\e}(\gamma(1))<0\}.
$$
By applying Mountain Pass Theorem \cite{AR}, we can see that there exists $u_{\e}\in \h\setminus\{0\}$ such that $J_{\e}(u_{\e})=c_{\e}$ and $J'_{\e}(u_{\e})=0$. In similar fashion, one can prove that also $J_{0}$ has a mountain pass geometry,  and we denote by $c_{V_{0}}$ the mountain pass level associated to $J_{0}$.\\
Now, let us introduce the Nehari manifold associated to (\ref{Pe}), that is
\begin{equation*}
\mathcal{N}_{\e}:= \{u\in \h \setminus \{0\} : \langle J_{\e}'(u), u \rangle =0\},
\end{equation*}
and we denote by $\mathcal{N}_{0}$ the Nehari manifold associated to \eqref{APe}.\\
It is standard to verify (see \cite{W}) that $c_{\e}$ can be also characterized as follows:
$$
c_{\e}=\inf_{u\in \h\setminus\{0\}} \sup_{t\geq 0} J_{\e}(t u)=\inf_{u\in \N_{\e}} J_{\e}(u);
$$

\noindent
Next, we prove the existence of a ground state solution to \eqref{APe}.
\begin{lem}\label{FS}
Let $(u_{n})\subset \mathcal{N}_{0}$ be a sequence satisfying $J_{0}(u_{n})\rightarrow c_{V_{0}}$. Then, up to subsequences, the following alternatives holds:
\begin{compactenum}[(i)]
\item $(u_{n})$ strongly converges in $H^{s}(\R^{3}, \R)$, 
\item there exists a sequence $(\tilde{y}_{n})\subset \R^{3}$ such that,  up to a subsequence, $v_{n}(x)=u_{n}(x+\tilde{y}_{n})$ converges strongly in $H^{s}(\R^{3}, \R)$.
\end{compactenum}
In particular, there exists a minimizer $w\in H^{s}(\R^{3}, \R)$ for $J_{0}$ with $J_{0}(w)=c_{V_{0}}$. 
\end{lem}
\begin{proof}
Since $J_{0}$ has a Mountain Pass geometry, we can use a version of the Mountain Pass Theorem without $(PS)$ condition (see \cite{W}), and we may suppose that  $(u_{n})$ is a $(PS)_{c_{V_{0}}}$ sequence for $J_{0}$. Arguing as in Lemma \ref{PSc}, it is easy to check that $(u_{n})$ is bounded in $H^{s}(\R^{3}, \R)$ so we may assume that $u_{n}\rightharpoonup u$ in $H^{s}(\R^{3}, \R)$. The weak convergence is enough to deduce that $J'_{0}(u)=0$. 
Now, we assume that $u\neq 0$. Since $u\in \mathcal{N}_{0}$, we can use $(f_3)$ and Fatou's Lemma to see that
\begin{align*}
c_{V_{0}}&\leq J_{0}(u)-\frac{1}{4}\langle J'_{0}(u), u\rangle \\
&=\frac{1}{4}\|u\|_{\mu}^{2}+\frac{1}{2}\int_{\R^{3}}\frac{1}{2} f(u^{2})u-F(u^{2})\, dx \\
&\leq \liminf_{n\rightarrow \infty} \left[J_{0}(u_{n})-\frac{1}{4}\langle J'_{0}(u), u\rangle\right]=c_{V_{0}},
\end{align*}
which implies that $J_{0}(u)=c_{V_{0}}$. 

Let us consider the case $u=0$. Since $c_{V_{0}}>0$ and $J_{0}$ is continuous, we can see that $\|u_{n}\|_{0}\not\rightarrow 0$. Then, in view of Lemma \ref{Lions} and $(f_1)$-$(f_2)$, it is standard to prove that there are a sequence $(y_{n})\subset \R^{3}$ and constants $R, \beta>0$ such that 
$$
\liminf_{n\rightarrow \infty} \int_{B_{R}(y_{n})}|u_{n}|^{4}dx\geq \beta>0.
$$
Let us define $v_{n}=u_{n}(\cdot+y_{n})$, and we note that $v_{n}$ has a nontrivial weak limit $v$ in $H^{s}(\R^{3}, \R)$. It is clear that also $(v_{n})$ is a $(PS)_{c_{V_{0}}}$ sequence for $J_{0}$, and arguing as before we can deduce that $J_{0}(v)=c_{V_{0}}$. In conclusion, problem \eqref{APe} admits a ground state solution.\\
Now, let $u$ be a ground state for \eqref{APe}. By using $\varphi=u^{-}$ as test function in $\langle J'_{0}(u), \varphi\rangle=0$, it is easy to check that $u\geq 0$ in $\R^{3}$. In particular, observing that $\phi_{u}^{t}\geq 0$ and $f$ has a subcritical growth, we can argue as in Proposition $5.1.1$ in \cite{DMV} to see that $u\in L^{\infty}(\R^{3}, \R)$. 
In particular, we have
\begin{align*}
\phi_{u}^{t}(x)&= \int_{|y-x|\geq 1} \frac{|u(y)|^{2}}{|x-y|^{3-2t}}dy+ \int_{|y-x|<1} \frac{|u(y)|^{2}}{|x-y|^{3-2t}}dy\\
&\leq \|u\|_{L^{2}(\R^{3})}^{2}+\|u\|^{2}_{L^{\infty}(\R^{3})} \int_{|y-x|<1} \frac{1}{|x-y|^{3-2t}}dy\leq C,
\end{align*}
so that $g(x)=f(u^{2})u-\mu u-\phi_{u}^{t}u\in L^{\infty}(\R^{3}, \R)$. By applying Proposition $2.9$ in \cite{S} we can deduce that $u\in C^{0, \gamma}(\R^{3}, \R)$ for some $0<\gamma<1$. By using maximum principle (see Corollary $3.4$ in \cite{FJ}) we can see that $u>0$ in $\R^{3}$. 
%Let $w$ be a solution to $-\Delta w=\mu u-\phi_{u}^{t}u+f(u^2)u\in C^{0, \alpha}(\R^{3}, \R)$. From the Schauder estimates for the Laplacian, we know that $w\in C^{2, \alpha}(\R^{3})$. It follows from $2s+\sigma>1$ that $(-\Delta)^{1-s}w\in C^{1, 2s+\alpha-1}$, and being $(-\Delta)^{s}(u-(-\Delta)^{1-s}w)=0$, we get that $u-(-\Delta)^{1-s}w$ is harmonic and $u$ has the same regularity of $(-\Delta)^{1-s}w$. Therefore $u\in C^{1, 2s+\alpha-1}(\R^{3}, \R)$. By using the following integral representation for the fractional Laplacian \cite{DPV}
%$$
%(-\Delta)^{s}u(x)=-\frac{C(3,s)}{2}\int_{\R^{3}} \frac{u(x+y)+u(x-y)-2u(x)}{|y|^{3+2s}} dy,
%$$
%we can see that if $u(x_{0})=0$ for some $x_{0}\in \R^{3}$, then 
%$$
%0>(-\Delta)^{s}u(x_{0})=-\mu u(x_{0})-\phi_{u}^{t}u(x_{0})+f(u(x_{0})^{2})u(x_{0})=0
%$$
%that is a contradiction.
Since $u\in C^{0, \alpha}(\R^{3}, \R)\cap L^{p}(\R^{3}, \R)$ for all $p\in [2, \infty)$, we can deduce that $u(x)\rightarrow 0$ as $|x|\rightarrow \infty$, so we can find $R>0$ such that $(-\Delta)^{s}u+\frac{V_{0}}{2}u\leq 0$ in $|x|>R$. By using Lemma $4.3$ in \cite{FQT} we know that there exists a positive function $w$ such that for $|x|>R$ (taking $R$ larger if necessary)  it holds $(-\Delta)^{s}w+\frac{V_{0}}{2}w\geq 0$ and $w(x)=\frac{C_{0}}{|x|^{3+2s}}$, for some $C_{0}>0$. 
 In view of the continuity of $u$ and $w$ there exists some constant $C_1>0$ such that $z:=u-C_{1}w\leq 0$ on $|x|=R$. Moreover, we can see that $(-\Delta)^{s}z+\frac{V_{0}}{2}z\geq 0$ in $|x|\geq R$. Then, it follows by the maximum principle that $z\leq 0$ in $|x|\geq R$, that is $0<u(x)\leq C_{1}w(x)\leq \frac{C_{2}}{|x|^{3+2s}}$ for all $|x|$ big enough. 

\end{proof}

Now we prove the following interesting relation between $c_{\e}$ and $c_{V_{0}}$.
\begin{lem}\label{AMlem1}
The numbers $c_{\e}$ and $c_{V_{0}}$ satisfy the following inequality
$$
\limsup_{\e\rightarrow 0} c_{\e}\leq c_{V_{0}}.
$$
\end{lem}
\begin{proof}
Let $w\in H^{s}(\R^{3}, \R)$ be a positive ground state to the autonomous problem \eqref{APe} (see \cite{Teng, ZDS}), so $J'_{0}(w)=0$ and $J_{0}(w)=c_{V_{0}}$. We recall that $w\in C^{0, \gamma}(\R^{3}, \R)\cap L^{\infty}(\R^{3}, \R)$.
Moreover, using Lemma $4.3$ in \cite{FQT}, it is easy to check that $w$ satisfies the following decay estimate: 
\begin{equation}\label{remdecay}
0<w(x)\leq \frac{C}{|x|^{3+2s}} \mbox{ for all } |x|>1.
\end{equation}
Let $\eta\in C^{\infty}_{c}(\R^{3}, [0,1])$ be a cut-off function such that $\eta=1$ in a neighborhood of zero $B_{\frac{\delta}{2}}$ and $\supp(\eta)\subset B_{\delta}\subset \Lambda$ for some $\delta>0$. 

Let us define $w_{\e}(x):=\eta_{\e}(x)w(x) e^{\imath A(0)\cdot x}$, with $\eta_{\e}(x)=\eta(\e x)$ for $\e>0$, and we observe that $|w_{\e}|=\eta_{\e}w$ and $w_{\e}\in \h$ in view of Lemma \ref{aux}. 
%By using the Dominated Convergence Theorem we can see that
%\begin{equation}\label{Poisslim}
%\lim_{\e\rightarrow 0} \int_{\R^{3}}\phi_{|w_{\e}|}^{t} |w_{\e}|^{2}dx=\int_{\R^{3}}\phi_{|w|}^{t} |w|^{2}dx.
%\end{equation}
Now we prove that
\begin{equation}\label{limwr}
\lim_{\e\rightarrow 0}\|w_{\e}\|^{2}_{\e}=\|w\|_{0}^{2}\in(0, \infty).
\end{equation}
Since it is clear that $\int_{\R^{3}} V_{\e}(x)|w_{\e}|^{2}dx\rightarrow \int_{\R^{3}} V_{0} |w|^{2}dx$, we only need to show that
\begin{equation}\label{limwr*}
\lim_{\e\rightarrow 0}[w_{\e}]^{2}_{A_{\e}}=[w]^{2}.
\end{equation}
By using Lemma $5$ in \cite{PP} we know that 
\begin{equation}\label{PPlem}
[\eta_{\e} w]\rightarrow [w] \mbox{ as } \e\rightarrow 0.
\end{equation}
On the other hand
\begin{align*}
[w_{\e}]_{A_{\e}}^{2}
&=\iint_{\R^{6}} \frac{|e^{\imath A(0)\cdot x}\eta_{\e}(x)w(x)-e^{\imath A_{\e}(\frac{x+y}{2})\cdot (x-y)}e^{\imath A(0)\cdot y} \eta_{\e}(y)w(y)|^{2}}{|x-y|^{3+2s}} dx dy \nonumber \\
&=[\eta_{\e} w]^{2}
+\iint_{\R^{6}} \frac{\eta_{\e}^2(y)w^2(y) |e^{\imath [A_{\e}(\frac{x+y}{2})-A(0)]\cdot (x-y)}-1|^{2}}{|x-y|^{3+2s}} dx dy\\
&\quad+2\Re \iint_{\R^{6}} \frac{(\eta_{\e}(x)w(x)-\eta_{\e}(y)w(y))\eta_{\e}(y)w(y)(1-e^{-\imath [A_{\e}(\frac{x+y}{2})-A(0)]\cdot (x-y)})}{|x-y|^{3+2s}} dx dy \\
&=: [\eta_{\e} w]^{2}+X_{\e}+2Y_{\e}.
\end{align*}
Then, in view of 
$|Y_{\e}|\leq [\eta_{\e} w] \sqrt{X_{\e}}$ and \eqref{PPlem}, it is suffices to prove that $X_{\e}\rightarrow 0$ as $\e\rightarrow 0$ to deduce that \eqref{limwr*} holds.
Let us note that for $0<\beta<\alpha/({1+\alpha-s})$, 
\begin{equation}\label{Ye}
\begin{split}
X_{\e}
&\leq \int_{\R^{3}} w^{2}(y) dy \int_{|x-y|\geq\e^{-\beta}} \frac{|e^{\imath [A_{\e}(\frac{x+y}{2})-A(0)]\cdot (x-y)}-1|^{2}}{|x-y|^{3+2s}} dx\\
&+\int_{\R^{3}} w^{2}(y) dy  \int_{|x-y|<\e^{-\beta}} \frac{|e^{\imath [A_{\e}(\frac{x+y}{2})-A(0)]\cdot (x-y)}-1|^{2}}{|x-y|^{3+2s}} dx \\
&=:X^{1}_{\e}+X^{2}_{\e}.
\end{split}
\end{equation}
Using $|e^{\imath t}-1|^{2}\leq 4$ and $w\in H^{s}(\R^{3}, \R)$, we get
\begin{equation}\label{Ye1}
X_{\e}^{1}\leq C \int_{\R^{3}} w^{2}(y) dy \int_{\e^{-\beta}}^\infty \rho^{-1-2s} d\rho\leq C \e^{2\beta s} \rightarrow 0.
\end{equation}
Since $|e^{\imath t}-1|^{2}\leq t^{2}$ for all $t\in \R$, $A\in C^{0,\alpha}(\R^3,\R^3)$ for $\alpha\in(0,1]$, and $|x+y|^{2}\leq 2(|x-y|^{2}+4|y|^{2})$, we have
\begin{equation}\label{Ye2}
\begin{split}
X^{2}_{\e}&
	\leq \int_{\R^{3}} w^{2}(y) dy  \int_{|x-y|<\e^{-\beta}} \frac{|A_{\e}\left(\frac{x+y}{2}\right)-A(0)|^{2} }{|x-y|^{3+2s-2}} dx \\
	&\leq C\e^{2\alpha} \int_{\R^{3}} w^{2}(y) dy  \int_{|x-y|<\e^{-\beta}} \frac{|x+y|^{2\alpha} }{|x-y|^{3+2s-2}} dx \\
	&\leq C\e^{2\alpha} \left(\int_{\R^{3}} w^{2}(y) dy  \int_{|x-y|<\e^{-\beta}} \frac{1 }{|x-y|^{3+2s-2-2\alpha}} dx\right.\\
	&\qquad\qquad+ \left. \int_{\R^{3}} |y|^{2\alpha} w^{2}(y) dy  \int_{|x-y|<\e^{-\beta}} \frac{1}{|x-y|^{3+2s-2}} dx\right) \\
	&=: C\e^{2\alpha} (X^{2, 1}_{\e}+X^{2, 2}_{\e}).
	\end{split}
	\end{equation}	
	Then
	\begin{equation}\label{Ye21}
	X^{2, 1}_{\e}
	= C  \int_{\R^{3}} w^{2}(y) dy \int_0^{\e^{-\beta}} \rho^{1+2\alpha-2s} d\rho
	\leq C\e^{-2\beta(1+\alpha-s)}.
	\end{equation}
	On the other hand, using \eqref{remdecay}, we infer that
	\begin{equation}\label{Ye22}
	\begin{split}
	 X^{2, 2}_{\e}
	 &\leq C  \int_{\R^{3}} |y|^{2\alpha} w^{2}(y) dy \int_0^{\e^{-\beta}}\rho^{1-2s} d\rho  \\
	&\leq C \e^{-2\beta(1-s)} \left[\int_{B_1(0)}  w^{2}(y) dy + \int_{B_1^c(0)} \frac{1}{|y|^{2(3+2s)-2\alpha}} dy \right]  \\
	&\leq C \e^{-2\beta(1-s)}.
	\end{split}
	\end{equation}
	Taking into account \eqref{Ye}, \eqref{Ye1}, \eqref{Ye2}, \eqref{Ye21} and \eqref{Ye22} we can conclude that $X_{\e}\rightarrow 0$. Therefore \eqref{limwr} holds.
	 Moreover, since $\eta_{\e}w$ strongly converges to $w$ in $H^{s}(\R^{3}, \R)$, we can use Lemma $2.4$-(5) in \cite{LZ} to see that
\begin{equation}\label{Poisslim}
\lim_{\e\rightarrow 0} \int_{\R^{3}}\phi_{|w_{\e}|}^{t} |w_{\e}|^{2}dx=\int_{\R^{3}}\phi_{w}^{t} w^{2}dx.
\end{equation}

Now, let $t_{\e}>0$ be the unique number such that 
\begin{equation*}
J_{\e}(t_{\e} w_{\e})=\max_{t\geq 0} J_{\e}(t w_{\e}).
\end{equation*}
Then $t_{\e}$ verifies 
\begin{equation}\label{AS1}
\|w_{\e}\|_{\e}^{2}+t_{\e}^{2}\int_{\R^{3}}\phi_{|w_{\e}|}^{t}|w_{\e}|^{2}dx=\int_{\R^{3}} g(\e x, t_{\e}^{2} |w_{\e}|^{2}) |w_{\e}|^{2}dx=\int_{\R^{3}} f(t_{\e}^{2} |w_{\e}|^{2}) |w_{\e}|^{2}dx
\end{equation}
where we used $supp(\eta)\subset \Lambda$ and $g=f$ on $\Lambda$.\\
Let us prove that $t_{\e}\rightarrow 1$ as $\e\rightarrow 0$. Using $\eta=1$ in $B_{\frac{\delta}{2}}$ and that $w$ is a continuous positive function, we can see that $(f_4)$ yields
$$
\frac{1}{t_{\e}^{2}}\|w_{\e}\|_{\e}^{2}+\int_{\R^{3}}\phi_{|w_{\e}|}^{t}|w_{\e}|^{2}dx\geq \frac{f(t_{\e}^{2}\alpha^{2}_{0})}{t_{\e}^{2}\alpha^{2}_{0}}\int_{B_{\frac{\delta}{2}}}|w|^{2}dx, 
$$
where $\alpha_{0}=\min_{\bar{B}_{\frac{\delta}{2}}} w>0$. So, if $t_{\e}\rightarrow \infty$ as $\e\rightarrow 0$, we can use $(f_3)$, \eqref{Poisslim} and \eqref{limwr} to deduce that $\int_{\R^{3}}\phi_{|w|,t}|w|^{2}dx= \infty$, that is a contradiction.
On the other hand, if $t_{\e}\rightarrow 0$ as $\e\rightarrow 0$, we can use the growth assumptions on $g$, \eqref{Poisslim}, \eqref{limwr} to infer that $\|w\|_{0}^{2}= 0$ which gives an absurd.
Therefore, $t_{\e}\rightarrow t_{0}\in (0, \infty)$ as $\e\rightarrow 0$.
Now, taking the limit as $\e\rightarrow 0$ in \eqref{AS1} and using \eqref{Poisslim}, \eqref{limwr}, we can deduce that 
\begin{equation}\label{AS2}
\frac{1}{t_{0}^{2}}\|w\|_{0}^{2}+\int_{\R^{3}}\phi_{|w|}^{t}|w|^{2}dx=\int_{\R^{3}} \frac{f(t_{0}^{2} |w|^{2})}{(t_{0}^{2}|w|^{2})} |w|^{4}dx.
\end{equation}
Then $t_{0}=1$ as a consequence of $w\in \mathcal{N}_{0}$ and $(f_4)$. By applying the Dominated Convergence Theorem, we know that 
$$
\int_{\R^{3}} F(|t_{\e}w_{\e}|^{2})\, dx\rightarrow \int_{\R^{3}} F(|w|^{2})\, dx
$$
so we have $\lim_{\e\rightarrow 0} J_{\e}(t_{\e} w_{\e})=J_{0}(w)=c_{V_{0}}$.
Since $c_{\e}\leq \max_{t\geq 0} J_{\e}(t w_{\e})=J_{\e}(t_{\e} w_{\e})$, we can infer  that
$\limsup_{\e\rightarrow 0} c_{\e}\leq c_{V_{0}}$.
\end{proof}

\noindent
Now, we prove the following useful compactness result.
\begin{lem}\label{prop3.3}
Let $\e_{n}\rightarrow 0$ and $(u_{n})\subset H^{s}_{\e_{n}}$ such that $J_{\e_{n}}(u_{n})=c_{\e_{n}}$ and $J'_{\e_{n}}(u_{n})=0$. Then there exists $(\tilde{y}_{n})\subset \R^{3}$ such that $v_{n}(x)=|u_{n}|(x+\tilde{y}_{n})$ has a convergent subsequence in $H^{s}(\R^{3}, \R)$. Moreover, up to a subsequence, $y_{n}=\e_{n} \tilde{y}_{n}\rightarrow y_{0}$ for some $y_{0}\in \Lambda$ such that $V(y_{0})=V_{0}$.
\end{lem}
\begin{proof}
Since $\langle J'_{\e_{n}}(u_{n}), u_{n}\rangle=0$, $J_{\e_{n}}(u_{n})= c_{\e_{n}}$ and using Lemma \ref{AMlem1}, we can see that $(u_{n})$ is bounded in $H^{s}_{\e_{n}}$. 
Then, there exists $C>0$ (independent of $n$) such that $\|u_{n}\|_{\e_{n}}\leq C$ for all $n\in \mathbb{N}$. Moreover, from Lemma \ref{DI}, we also know that $(|u_{n}|)$ is bounded in $H^{s}(\R^{3}, \R)$.\\
Now, we prove that there exist a sequence $(\tilde{y}_{n})\subset \R^{3}$, and constants $R>0$ and $\gamma>0$ such that
\begin{equation}\label{sacchi}
\liminf_{n\rightarrow \infty}\int_{B_{R}(\tilde{y}_{n})} |u_{n}|^{2} \, dx\geq \gamma>0.
\end{equation}
Assume by contradiction \eqref{sacchi} does not hold, so that, for all $R>0$ we get
$$
\lim_{n\rightarrow \infty}\sup_{y\in \R^{3}}\int_{B_{R}(y)} |u_{n}|^{2} \, dx=0.
$$
By using the boundedness of $(|u_{n}|)$ and Lemma \ref{Lions}, we know that $|u_{n}|\rightarrow 0$ in $L^{q}(\R^{3}, \R)$ for any $q\in (2, 2^{*}_{s})$. 
This fact and $(g_1)$ and $(g_2)$ yield
\begin{align}\label{glimiti}
\lim_{n\rightarrow \infty}\int_{\R^{3}} g(\e_{n} x, |u_{n}|^{2}) |u_{n}|^{2} \,dx=0
= \lim_{n\rightarrow \infty}\int_{\R^{3}} G(\e_{n} x, |u_{n}|^{2}) \, dx.
\end{align}
On the other hand, $|u_{n}|\rightarrow 0$ in $L^{\frac{12}{3+2t}}(\R^{3}, \R)$, so by Lemma \ref{poisson}-$(4)$ we deduce that
\begin{align}\label{CSSq1}
\int_{\R^{3}} \phi_{|u_{n}|}^{t}|u_{n}|^{2}dx\rightarrow 0.
\end{align}
Taking into account $\langle J'_{\e_{n}}(u_{n}), u_{n}\rangle=0$, \eqref{glimiti} and \eqref{CSSq1} we can  infer that 
$\|u_{n}\|_{\e_{n}}\rightarrow 0$ as $n\rightarrow \infty$. This is impossible because $(g_1)$, $(g_2)$ and $\langle J'_{\e_{n}}(u_{n}), u_{n}\rangle=0$ imply that there exists $\alpha_{0}>0$ such that $\|u_{n}\|^{2}_{\e_{n}}\geq \alpha_{0}$ for all $n\in \mathbb{N}$.
Now, we set $v_{n}(x)=|u_{n}|(x+\tilde{y}_{n})$. Then $(v_{n})$ is bounded in $H^{s}(\R^{3}, \R)$, and we may assume that 
$v_{n}\rightharpoonup v\not\equiv 0$ in $H^{s}(\R^{3}, \R)$  as $n\rightarrow \infty$.
Fix $t_{n}>0$ such that $\tilde{v}_{n}=t_{n} v_{n}\in \mathcal{N}_{0}$. In view of Lemma \ref{DI}, we have 
$$
c_{V_{0}}\leq J_{0}(\tilde{v}_{n})\leq \max_{t\geq 0}J_{\e_{n}}(tv_{n})= J_{\e_{n}}(u_{n})
%c_{V_{0}}\leq J_{0}(\tilde{v}_{n})\leq \max_{t\geq 0}J_{\e_{n}}(tv_{n})= J_{\e_{n}}(u_{n})= c_{V_{0}}+o_{n}(1)
$$
which together with Lemma \ref{AMlem1} yields $J_{0}(\tilde{v}_{n})\rightarrow c_{V_{0}}$. Then, $\tilde{v}_{n}\nrightarrow 0$ in $H^{s}(\R^{3}, \R)$.
Since $(v_{n})$ and $(\tilde{v}_{n})$ are bounded in $H^{s}(\R^{3}, \R)$ and $\tilde{v}_{n}\nrightarrow 0$  in $H^{s}(\R^{3}, \R)$, we deduce that $t_{n}\rightarrow t^{*}>0$. 
%Indeed $t^{*}>0$ since $\tilde{v}_{n}\nrightarrow 0$  in $H^{s}(\R^{3}, \R)$. 
From the uniqueness of the weak limit, we can deduce that $\tilde{v}_{n}\rightharpoonup \tilde{v}=t^{*}v\not\equiv 0$ in $H^{s}(\R^{3}, \R)$,
and by using Lemma \ref{FS}, we can infer that 
\begin{equation}\label{elena}
\tilde{v}_{n}\rightarrow \tilde{v} \mbox{ in } H^{s}(\R^{3}, \R).
\end{equation} 
Therefore, $v_{n}\rightarrow v$ in $H^{s}(\R^{3}, \R)$ as $n\rightarrow \infty$.

Now, we define $y_{n}=\e_{n}\tilde{y}_{n}$ and we show that $(y_{n})$ admits a subsequence, still denoted by $y_{n}$, such that $y_{n}\rightarrow y_{0}$ for some $y_{0}\in \Lambda$ such that $V(y_{0})=V_{0}$. Firstly, we prove that $(y_{n})$ is bounded. Assume by contradiction that, up to a subsequence, $|y_{n}|\rightarrow \infty$ as $n\rightarrow \infty$. Take $R>0$ such that $\Lambda \subset B_{R}(0)$. Since we may suppose that  $|y_{n}|>2R$, we have that for any $z\in B_{R/\e_{n}}$ 
$$
|\e_{n}z+y_{n}|\geq |y_{n}|-|\e_{n}z|>R.
$$
Now, by using $(u_{n})\subset \N_{\e_{n}}$, $(V_{1})$, Lemma \ref{DI}, Lemma \ref{poisson} and the change of variable $x\mapsto z+\tilde{y}_{n}$ we observe that 
\begin{align}\label{pasq}
[v_{n}]^{2}+\int_{\R^{3}} V_{0} v_{n}^{2}\, dx &[v_{n}]^{2}+\int_{\R^{3}} V_{0} v_{n}^{2}\, dx+\int_{\R^{3}} \phi_{|v_{n}|}^{t}v_{n}^{2}\, dx\\
&\leq \int_{\R^{3}} g(\e_{n} x+y_{n}, |v_{n}|^{2}) |v_{n}|^{2} \, dx \nonumber\\
&\leq \int_{B_{\frac{R}{\e_{n}}}(0)} \tilde{f}(|v_{n}|^{2}) |v_{n}|^{2} \, dx+\int_{\R^{3}\setminus B_{\frac{R}{\e_{n}}}(0)} f(|v_{n}|^{2}) |v_{n}|^{2} \, dx.
\end{align}
Since $v_{n}\rightarrow v$ in $H^{s}(\R^{3}, \R)$ as $n\rightarrow \infty$ and $\tilde{f}(t)\leq \frac{V_{0}}{k}$, we can see that (\ref{pasq}) yields
$$
\min\left\{1, V_{0}\left(1-\frac{1}{k}\right) \right\} \left([v_{n}]^{2}+\int_{\R^{3}} |v_{n}|^{2}\, dx\right)=o_{n}(1),
$$
that is $v_{n}\rightarrow 0$ in $H^{s}(\R^{3}, \R)$ and this gives a contradiction. Thus, $(y_{n})$ is bounded and we may assume that $y_{n}\rightarrow y_{0}\in \R^{3}$. If $y_{0}\notin \overline{\Lambda}$, we can proceed as before to deduce that $v_{n}\rightarrow 0$ in $H^{s}(\R^{3}, \R)$. Therefore $y_{0}\in \overline{\Lambda}$. We observe that if $V(y_{0})=V_{0}$, then $y_{0}\notin \partial \Lambda$ in view of $(V_2)$. Then, it is enough to verify that $V(y_{0})=V_{0}$. Otherwise, we suppose that $V(y_{0})>V_{0}$, and
putting together (\ref{elena}), Fatou's Lemma, the invariance of  $\R^{3}$ by translations, Lemma \ref{DI} and Lemma \ref{AMlem1}, we have
\begin{align*}
c_{V_{0}}=J_{0}(\tilde{v})&<\frac{1}{2}[\tilde{v}]^{2}+\frac{1}{2}\int_{\R^{3}} V(y_{0})\tilde{v}^{2} \, dx+\frac{1}{4}\int_{\R^{3}}\phi_{|\tilde{v}|}^{t}\tilde{v}^{2}dx-\frac{1}{2}\int_{\R^{3}} F(|\tilde{v}|^{2})\, dx\\
&\leq \liminf_{n\rightarrow \infty}\left[\frac{1}{2}[\tilde{v}_{n}]^{2}+\frac{1}{2}\int_{\R^{3}} V(\e_{n}x+y_{n}) |\tilde{v}_{n}|^{2} \, dx+\frac{1}{4}\int_{\R^{3}}\phi_{|\tilde{v}_{n}|}^{t}|\tilde{v}_{n}|^{2}dx-\frac{1}{2}\int_{\R^{3}} F(|\tilde{v}_{n}|^{2})\, dx  \right] \\
&\leq \liminf_{n\rightarrow \infty}\left[\frac{t_{n}^{2}}{2}[|u_{n}|]^{2}+\frac{t_{n}^{2}}{2}\int_{\R^{3}} V(\e_{n}z) |u_{n}|^{2} \, dz+ \frac{t_{n}^{4}}{4}\int_{\R^{3}}\phi_{|u_{n}|}^{t}|u_{n}|^{2}dx -\frac{1}{2}\int_{\R^{3}} F(|t_{n} u_{n}|^{2})\, dz  \right] \\
&\leq \liminf_{n\rightarrow \infty} J_{\e_{n}}(t_{n} u_{n}) \leq \liminf_{n\rightarrow \infty} J_{\e_{n}}(u_{n})\leq c_{V_{0}}
\end{align*}
which is a contradiction. This ends the proof of this lemma.

\end{proof}

\section{Proof of Theorem \ref{thm1}}
This section is devoted to the proof of the main theorem of this work.
Firstly, we prove the following lemma which plays a fundamental role to show that the solutions of \eqref{Pe} are indeed solutions  to \eqref{P}.
\begin{lem}\label{moser} 
Let $\e_{n}\rightarrow 0$ and $u_{n}\in H^{s}_{\e_{n}}$ be a solution to \eqref{MPe}. 
Then $v_{n}=|u_{n}|(\cdot+\tilde{y}_{n})$ satisfies $v_{n}\in L^{\infty}(\R^{3},\R)$ and there exists $C>0$ such that 
$$
\|v_{n}\|_{L^{\infty}(\R^{3})}\leq C \mbox{ for all } n\in \mathbb{N},
$$
where $\tilde{y}_{n}$ is given by Lemma \ref{prop3.3}.
Moreover it holds
$$
\lim_{|x|\rightarrow \infty} v_{n}(x)=0 \mbox{ uniformly in } n\in \mathbb{N}.
$$
\end{lem}
\begin{proof}
For any $L>0$, we define $u_{L,n}:=\min\{|u_{n}|, L\}\geq 0$ and we set $v_{L, n}=u_{L,n}^{2(\beta-1)}u_{n}$ where $\beta>1$ will be chosen later.
Taking $v_{L, n}$ as test function in (\ref{MPe}) we can see that
\begin{align}\label{conto1FF}
&\Re\left(\iint_{\R^{6}} \frac{(u_{n}(x)-u_{n}(y)e^{\imath A(\frac{x+y}{2})\cdot (x-y)})}{|x-y|^{3+2s}} \overline{(u_{n}u_{L,n}^{2(\beta-1)}(x)-u_{n}u_{L,n}^{2(\beta-1)}(y)e^{\imath A(\frac{x+y}{2})\cdot (x-y)})} \, dx dy\right)   \nonumber \\
&=-\int_{\R^{3}}\phi_{|u_{n}|}^{t}|u_{n}|^{2}u_{L,n}^{2(\beta-1)}dx+\int_{\R^{3}} g(\e x, |u_{n}|^{2}) |u_{n}|^{2}u_{L,n}^{2(\beta-1)}  \,dx-\int_{\R^{3}} V(\e x) |u_{n}|^{2} u_{L,n}^{2(\beta-1)} \, dx.
\end{align}
Let us observe that
\begin{align*}
&\Re\left[(u_{n}(x)-u_{n}(y)e^{\imath A(\frac{x+y}{2})\cdot (x-y)})\overline{(u_{n}u_{L,n}^{2(\beta-1)}(x)-u_{n}u_{L,n}^{2(\beta-1)}(y)e^{\imath A(\frac{x+y}{2})\cdot (x-y)})}\right] \\
&=\Re\Bigl[|u_{n}(x)|^{2}v_{L}^{2(\beta-1)}(x)-u_{n}(x)\overline{u_{n}(y)} u_{L,n}^{2(\beta-1)}(y)e^{-\imath A(\frac{x+y}{2})\cdot (x-y)}-u_{n}(y)\overline{u_{n}(x)} u_{L,n}^{2(\beta-1)}(x) e^{\imath A(\frac{x+y}{2})\cdot (x-y)} \\
&+|u_{n}(y)|^{2}u_{L,n}^{2(\beta-1)}(y) \Bigr] \\
&\geq (|u_{n}(x)|^{2}u_{L,n}^{2(\beta-1)}(x)-|u_{n}(x)||u_{n}(y)|u_{L,n}^{2(\beta-1)}(y)-|u_{n}(y)||u_{n}(x)|u_{L,n}^{2(\beta-1)}(x)+|u_{n}(y)|^{2}u^{2(\beta-1)}_{L,n}(y) \\
&=(|u_{n}(x)|-|u_{n}(y)|)(|u_{n}(x)|u_{L,n}^{2(\beta-1)}(x)-|u_{n}(y)|u_{L,n}^{2(\beta-1)}(y)),
\end{align*}
from which we deduce that
\begin{align}\label{realeF}
&\Re\left(\iint_{\R^{6}} \frac{(u_{n}(x)-u_{n}(y)e^{\imath A(\frac{x+y}{2})\cdot (x-y)})}{|x-y|^{3+2s}} \overline{(u_{n}u_{L,n}^{2(\beta-1)}(x)-u_{n}u_{L,n}^{2(\beta-1)}(y)e^{\imath A(\frac{x+y}{2})\cdot (x-y)})} \, dx dy\right) \nonumber\\
&\geq \iint_{\R^{6}} \frac{(|u_{n}(x)|-|u_{n}(y)|)}{|x-y|^{3+2s}} (|u_{n}(x)|u_{L,n}^{2(\beta-1)}(x)-|u_{n}(y)|u_{L,n}^{2(\beta-1)}(y))\, dx dy.
\end{align}
For all $t\geq 0$, let us define
\begin{equation*}
\gamma(t)=\gamma_{L, \beta}(t)=t t_{L}^{2(\beta-1)}
\end{equation*}
where  $t_{L}=\min\{t, L\}$. 
Let us observe that, since $\gamma$ is an increasing function, then it holds
\begin{align*}
(a-b)(\gamma(a)- \gamma(b))\geq 0 \quad \mbox{ for any } a, b\in \R.
\end{align*}
Let us define the functions 
\begin{equation*}
\Lambda(t)=\frac{|t|^{2}}{2} \quad \mbox{ and } \quad \Gamma(t)=\int_{0}^{t} (\gamma'(\tau))^{\frac{1}{2}} d\tau. 
\end{equation*}
and we note that
\begin{equation}\label{Gg}
\Lambda'(a-b)(\gamma(a)-\gamma(b))\geq |\Gamma(a)-\Gamma(b)|^{2} \mbox{ for any } a, b\in\R. 
\end{equation}
Indeed, for any $a, b\in \R$ such that $a<b$, and using the Jensen inequality we have
\begin{align*}
\Lambda'(a-b)(\gamma(a)-\gamma(b))=(a-b)\int_{b}^{a} \gamma'(t)dt=(a-b)\int_{b}^{a} (\Gamma'(t))^{2}dt\geq \left(\int_{b}^{a} \Gamma'(t) dt\right)^{2}=(\Gamma(a)-\Gamma(b))^{2}.
\end{align*}
In view of \eqref{Gg} we can deduce that
\begin{align}\label{Gg1}
|\Gamma(|u_{n}(x)|)- \Gamma(|u_{n}(y)|)|^{2} \leq (|u_{n}(x)|- |u_{n}(y)|)((|u_{n}|u_{L,n}^{2(\beta-1)})(x)- (|u_{n}|u_{L,n}^{2(\beta-1)})(y)). 
\end{align}
Putting together \eqref{realeF} and \eqref{Gg1} we have
\begin{align}\label{conto1FFF}
\Re\left(\iint_{\R^{6}} \frac{(u_{n}(x)-u_{n}(y)e^{\imath A(\frac{x+y}{2})\cdot (x-y)})}{|x-y|^{3+2s}} \overline{(u_{n}u_{L,n}^{2(\beta-1)}(x)-u_{n}u_{L,n}^{2(\beta-1)}(y)e^{\imath A(\frac{x+y}{2})\cdot (x-y)})} \, dx dy\right) 
\geq [\Gamma(|u_{n}|)]^{2}.
\end{align}
%\begin{align}\label{BMS}
%&[\Gamma(v_{n})]^{2}+\int_{\R^{N}} V_{0}|v_{n}|^{2}v_{L, n}^{2(\beta-1)} dx \nonumber \\
%&\leq \iint_{\R^{2N}} \frac{(v_{n}(x)- v_{n}(y))}{|x-y|^{N+2s}} ((v_{n}v_{L, n}^{2(\beta-1)})(x)-(v_{n} v_{L,n}^{2(\beta-1)})(y)) \,dx dy +\int_{\R^{N}} V(\e_{n}x+\e_{n}\tilde{y}_{n})|v_{n}|^{2}v_{L,n}^{2(\beta-1)} dx \nonumber\\
%&\leq \int_{\R^{N}} g(\e_{n}x+\e_{n}\tilde{y}_{n}, v^{2}_{n}) v^{2}_{n} v_{L,n}^{2(\beta-1)} dx.
%\end{align}
Being $\Gamma(|u_{n}|)\geq \frac{1}{\beta} |u_{n}| u_{L,n}^{\beta-1}$ and recalling that $\mathcal{D}^{s,2}(\R^{3}, \R)\subset L^{\2}(\R^{3}, \R)$ (see \cite{DPV}), we get 
\begin{equation}\label{SS1}
[\Gamma(|u_{n}|)]^{2}\geq S_{*} \|\Gamma(|u_{n}|)\|^{2}_{L^{\2}(\R^{3})}\geq \left(\frac{1}{\beta}\right)^{2} S_{*}\||u_{n}| u_{L,n}^{\beta-1}\|^{2}_{L^{\2}(\R^{3})}.
\end{equation}
Taking into account \eqref{conto1FF}, \eqref{conto1FFF} and \eqref{SS1} we obtain
\begin{align}\label{BMS}
\left(\frac{1}{\beta}\right)^{2} S_{*}\||u_{n}| u_{L,n}^{\beta-1}\|^{2}_{L^{\2}(\R^{3})}+\int_{\R^{3}} V(\e x)|u_{n}|^{2}u_{L,n}^{2(\beta-1)} dx\leq \int_{\R^{3}} g(\e_{n}x, |u_{n}|^{2}) |u_{n}|^{2} u_{L,n}^{2(\beta-1)} dx.
\end{align}
On the other hand, from the assumptions $(g_1)$ and $(g_2)$, for any $\xi>0$ there exists $C_{\xi}>0$ such that
\begin{equation}\label{SS2}
g(\e x, t^{2})t^{2}\leq \xi |t|^{2}+C_{\xi}|t|^{\2} \mbox{ for all } t\in \R.
\end{equation}
Taking $\xi\in (0, V_{0})$ and using \eqref{BMS} and \eqref{SS2} and Lemma \ref{poisson} we can obtain that
\begin{equation}\label{simo1}
\|w_{L,n}\|_{L^{\2}(\R^{3})}^{2}\leq C \beta^{2} \int_{\R^{3}} |u_{n}|^{\2}u_{L,n}^{2(\beta-1)},
\end{equation}
where  we set $w_{L,n}:=|u_{n}| u_{L,n}^{\beta-1}$.
Now, take $\beta=\frac{\2}{2}$ and fix $R>0$. Observing that $0\leq u_{L,n}\leq |u_{n}|$ and applying H\"older inequality we have
\begin{align}\label{simo2}
\int_{\R^{3}} |u_{n}|^{\2}u_{L,n}^{2(\beta-1)}dx&=\int_{\R^{3}} |u_{n}|^{\2-2} |u_{n}|^{2} u_{L,n}^{\2-2}dx \nonumber\\
&=\int_{\R^{3}} |u_{n}|^{\2-2} (|u_{n}| u_{L,n}^{\frac{\2-2}{2}})^{2}dx \nonumber\\
&\leq \int_{\{|u_{n}|<R\}} R^{\2-2} |u_{n}|^{\2} dx+\int_{\{|u_{n}|>R\}} |u_{n}|^{\2-2} (|u_{n}| u_{L,n}^{\frac{\2-2}{2}})^{2}dx \nonumber\\
&\leq \int_{\{|u_{n}|<R\}} R^{\2-2} |u_{n}|^{\2} dx+\left(\int_{\{|u_{n}|>R\}} |u_{n}|^{\2} dx\right)^{\frac{\2-2}{\2}} \left(\int_{\R^{3}} (|u_{n}| u_{L,n}^{\frac{\2-2}{2}})^{\2}dx\right)^{\frac{2}{\2}}.
\end{align}
Since $(|u_{n}|)$ is bounded in $H^{s}(\R^{3}, \R)$, we can choose $R$ sufficiently large such that
\begin{equation}\label{simo3}
\left(\int_{\{|u_{n}|>R\}} |u_{n}|^{\2} dx\right)^{\frac{\2-2}{\2}}\leq  \frac{1}{2\beta^{2}}.
\end{equation}
In view of \eqref{simo1}, \eqref{simo2} and \eqref{simo3} we can infer
\begin{equation*}
\left(\int_{\R^{3}} (|u_{n}| u_{L,n}^{\frac{\2-2}{2}})^{\2} \right)^{\frac{2}{\2}}\leq C\beta^{2} \int_{\R^{3}} R^{\2-2} |u_{n}|^{\2} dx<\infty
\end{equation*}
and letting the limit as $L\rightarrow \infty$ we obtain $|u_{n}|\in L^{\frac{(\2)^{2}}{2}}(\R^{3},\R)$.\\
Now, using $0\leq u_{L,n}\leq |u_{n}|$ and taking the limit as $L\rightarrow \infty$ in \eqref{simo1} we have
\begin{equation*}
\||u_{n}|\|_{L^{\beta\2}(\R^{3})}^{2\beta}\leq C \beta^{2} \int_{\R^{3}} |u_{n}|^{\2+2(\beta-1)},
\end{equation*}
from which we deduce that
\begin{equation*}
\left(\int_{\R^{3}} |u_{n}|^{\beta\2} dx\right)^{\frac{1}{(\beta-1)\2}}\leq C \beta^{\frac{1}{\beta-1}} \left(\int_{\R^{3}} |u_{n}|^{\2+2(\beta-1)}\right)^{\frac{1}{2(\beta-1)}}.
\end{equation*}
For $m\geq 1$ we define $\beta_{m+1}$ inductively so that $\2+2(\beta_{m+1}-1)=\2 \beta_{m}$ and $\beta_{1}=\frac{\2}{2}$. \\
Then we can see that
\begin{equation*}
\left(\int_{\R^{3}} |u_{n}|^{\beta_{m+1}\2} dx\right)^{\frac{1}{(\beta_{m+1}-1)\2}}\leq C \beta_{m+1}^{\frac{1}{\beta_{m+1}-1}} \left(\int_{\R^{3}} |u_{n}|^{\2\beta_{m}}\right)^{\frac{1}{\2(\beta_{m}-1)}}.
\end{equation*}
Let us define
$$
D_{m}=\left(\int_{\R^{3}} |u_{n}|^{\2\beta_{m}}\right)^{\frac{1}{\2(\beta_{m}-1)}},
$$
and by using an iteration argument, we can find $C_{0}>0$ independent of $m$ such that 
$$
D_{m+1}\leq \prod_{k=1}^{m} C \beta_{k+1}^{\frac{1}{\beta_{k+1}-1}}  D_{1}\leq C_{0} D_{1}.
$$
Passing to the limit as $m\rightarrow \infty$ we find 
\begin{equation}\label{UBu}
\|u_{n}\|_{L^{\infty}(\R^{3})}\leq C_{0}D_{1}=:K \mbox{ for all } n\in \mathbb{N}.
\end{equation}
Clearly, by interpolation, we can deduce that $(|u_{n}|)$ strongly converges in $L^{r}(\R^{3}, \R)$ for all $r\in (2, \infty)$.
From the growth assumptions on $g$, also $g(\e x, |u_{n}|^{2})|u_{n}|$ strongly converges  in the same Lebesgue spaces. 

In what follows, we show that $|u_{n}|$ is a weak subsolution to 
\begin{equation}\label{Kato0}
\left\{
\begin{array}{ll}
(-\Delta)^{s}v+V_{0} v=g(\e x, v^{2})v &\mbox{ in } \R^{3} \\
v\geq 0 \quad \mbox{ in } \R^{3}.
\end{array}
\right.
\end{equation}
Fix $\varphi\in C^{\infty}_{c}(\R^{3}, \R)$ such that $\varphi\geq 0$, and we take $\psi_{\delta, n}=\frac{u_{n}}{u_{\delta, n}}\varphi$ as test function in \eqref{Pe}, where we set $u_{\delta,n}=\sqrt{|u_{n}|^{2}+\delta^{2}}$ for $\delta>0$. We note that $\psi_{\delta, n}\in H^{s}_{\e_{n}}$ for all $\delta>0$ and $n\in \mathbb{N}$. Indeed $\int_{\R^{3}} V(\e_{n} x) |\psi_{\delta,n}|^{2} dx\leq \int_{\supp(\varphi)} V(\e_{n} x)\varphi^{2} dx<\infty$. 
%For simplicity, we write $u$, $u_{\delta}$ and $\psi_{\delta}$ instead of $u_{n}$, $u_{\delta, n}$ and $\psi_{\delta,n}$. 
On the other hand, we can observe
\begin{align*}
\psi_{\delta,n}(x)-\psi_{\delta,n}(y)e^{\imath A_{\e}(\frac{x+y}{2})\cdot (x-y)}&=\left(\frac{u_{n}(x)}{u_{\delta,n}(x)}\right)\varphi(x)-\left(\frac{u_{n}(y)}{u_{\delta,n}(y)}\right)\varphi(y)e^{\imath A_{\e}(\frac{x+y}{2})\cdot (x-y)}\\
&=\left[\left(\frac{u_{n}(x)}{u_{\delta,n}(x)}\right)-\left(\frac{u_{n}(y)}{u_{\delta,n}(x)}\right)e^{\imath A_{\e}(\frac{x+y}{2})\cdot (x-y)}\right]\varphi(x) \\
&+\left[\varphi(x)-\varphi(y)\right] \left(\frac{u_{n}(y)}{u_{\delta,n}(x)}\right) e^{\imath A_{\e}(\frac{x+y}{2})\cdot (x-y)} \\
&+\left(\frac{u_{n}(y)}{u_{\delta,n}(x)}-\frac{u_{n}(y)}{u_{\delta,n}(y)}\right)\varphi(y) e^{\imath A_{\e}(\frac{x+y}{2})\cdot (x-y)}
\end{align*}
which gives
\begin{align*}
&|\psi_{\delta,n}(x)-\psi_{\delta,n}(y)e^{\imath A_{\e}(\frac{x+y}{2})\cdot (x-y)}|^{2} \\
&\leq \frac{4}{\delta^{2}}|u_{n}(x)-u_{n}(y)e^{\imath A_{\e}(\frac{x+y}{2})\cdot (x-y)}|^{2}\|\varphi\|^{2}_{L^{\infty}(\R^{3})} +\frac{4}{\delta^{2}}|\varphi(x)-\varphi(y)|^{2} \|u_{n}\|^{2}_{L^{\infty}(\R^{3})} \\
&+\frac{4}{\delta^{4}} \|u_{n}\|^{2}_{L^{\infty}(\R^{3})} \|\varphi\|^{2}_{L^{\infty}(\R^{3})} |u_{\delta,n}(y)-u_{\delta,n}(x)|^{2} \\
&\leq \frac{4}{\delta^{2}}|u_{n}(x)-u_{n}(y)e^{\imath A_{\e}(\frac{x+y}{2})\cdot (x-y)}|^{2}\|\varphi\|^{2}_{L^{\infty}(\R^{3})} +\frac{4K^{2}}{\delta^{2}}|\varphi(x)-\varphi(y)|^{2} \\
&+\frac{4K^{2}}{\delta^{4}} \|\varphi\|^{2}_{L^{\infty}(\R^{3})} ||u_{n}(y)|-|u_{n}(x)||^{2} 
\end{align*}
where we used 
$$
|z+w+k|^{2}\leq 4(|z|^{2}+|w|^{2}+|k|^{2}) \quad \forall z,w,k\in \C,
$$
$|e^{\imath t}|=1$ for all $t\in \R$, $u_{\delta,n}\geq \delta$, $|\frac{u_{n}}{u_{\delta,n}}|\leq 1$, \eqref{UBu} and the following inequality
$$
|\sqrt{|z|^{2}+\delta^{2}}-\sqrt{|w|^{2}+\delta^{2}}|\leq ||z|-|w|| \quad \forall z, w\in \C.
$$
Since $u_{n}\in H^{s}_{\e_{n}}$, $|u_{n}|\in H^{s}(\R^{3}, \R)$ (by Lemma \ref{DI}) and $\varphi\in C^{\infty}_{c}(\R^{3}, \R)$, we deduce that $\psi_{\delta,n}\in H^{s}_{\e_{n}}$.\\
Therefore
\begin{align}\label{Kato1}
&\Re\left[\iint_{\R^{6}} \frac{(u_{n}(x)-u_{n}(y)e^{\imath A_{\e}(\frac{x+y}{2})\cdot (x-y)})}{|x-y|^{3+2s}} \left(\frac{\overline{u_{n}(x)}}{u_{\delta,n}(x)}\varphi(x)-\frac{\overline{u_{n}(y)}}{u_{\delta,n}(y)}\varphi(y)e^{-\imath A_{\e}(\frac{x+y}{2})\cdot (x-y)}  \right) dx dy\right] \nonumber\\
&+\int_{\R^{3}} V(\e x)\frac{|u_{n}|^{2}}{u_{\delta,n}}\varphi dx+\int_{\R^{3}} \phi_{|u_{n}|}^{t} \frac{|u_{n}|^{2}}{u_{\delta,n}}\varphi dx=\int_{\R^{3}} g(\e x, |u_{n}|^{2})\frac{|u_{n}|^{2}}{u_{\delta,n}}\varphi dx.
\end{align}
Since $\Re(z)\leq |z|$ for all $z\in \C$ and  $|e^{\imath t}|=1$ for all $t\in \R$, we get
\begin{align}\label{alves1}
&\Re\left[(u_{n}(x)-u_{n}(y)e^{\imath A_{\e}(\frac{x+y}{2})\cdot (x-y)}) \left(\frac{\overline{u_{n}(x)}}{u_{\delta,n}(x)}\varphi(x)-\frac{\overline{u_{n}(y)}}{u_{\delta,n}(y)}\varphi(y)e^{-\imath A_{\e}(\frac{x+y}{2})\cdot (x-y)}  \right)\right] \nonumber\\
&=\Re\left[\frac{|u_{n}(x)|^{2}}{u_{\delta,n}(x)}\varphi(x)+\frac{|u_{n}(y)|^{2}}{u_{\delta,n}(y)}\varphi(y)-\frac{u_{n}(x)\overline{u_{n}(y)}}{u_{\delta,n}(y)}\varphi(y)e^{-\imath A_{\e}(\frac{x+y}{2})\cdot (x-y)} -\frac{u_{n}(y)\overline{u_{n}(x)}}{u_{\delta,n}(x)}\varphi(x)e^{\imath A_{\e}(\frac{x+y}{2})\cdot (x-y)}\right] \nonumber \\
&\geq \left[\frac{|u_{n}(x)|^{2}}{u_{\delta,n}(x)}\varphi(x)+\frac{|u_{n}(y)|^{2}}{u_{\delta,n}(y)}\varphi(y)-|u_{n}(x)|\frac{|u_{n}(y)|}{u_{\delta,n}(y)}\varphi(y)-|u_{n}(y)|\frac{|u_{n}(x)|}{u_{\delta,n}(x)}\varphi(x) \right].
\end{align}
Now, we can note that
\begin{align}\label{alves2}
&\frac{|u_{n}(x)|^{2}}{u_{\delta,n}(x)}\varphi(x)+\frac{|u_{n}(y)|^{2}}{u_{\delta,n}(y)}\varphi(y)-|u_{n}(x)|\frac{|u_{n}(y)|}{u_{\delta,n}(y)}\varphi(y)-|u_{n}(y)|\frac{|u_{n}(x)|}{u_{\delta,n}(x)}\varphi(x) \nonumber\\
&=  \frac{|u_{n}(x)|}{u_{\delta,n}(x)}(|u_{n}(x)|-|u_{n}(y)|)\varphi(x)-\frac{|u_{n}(y)|}{u_{\delta,n}(y)}(|u_{n}(x)|-|u_{n}(y)|)\varphi(y) \nonumber\\
&=\left[\frac{|u_{n}(x)|}{u_{\delta,n}(x)}(|u_{n}(x)|-|u_{n}(y)|)\varphi(x)-\frac{|u_{n}(x)|}{u_{\delta,n}(x)}(|u_{n}(x)|-|u_{n}(y)|)\varphi(y)\right] \nonumber\\
&+\left(\frac{|u_{n}(x)|}{u_{\delta,n}(x)}-\frac{|u_{n}(y)|}{u_{\delta,n}(y)} \right) (|u_{n}(x)|-|u_{n}(y)|)\varphi(y) \nonumber\\
&=\frac{|u_{n}(x)|}{u_{\delta,n}(x)}(|u_{n}(x)|-|u_{n}(y)|)(\varphi(x)-\varphi(y)) +\left(\frac{|u_{n}(x)|}{u_{\delta,n}(x)}-\frac{|u_{n}(y)|}{u_{\delta,n}(y)} \right) (|u_{n}(x)|-|u_{n}(y)|)\varphi(y) \nonumber\\
&\geq \frac{|u_{n}(x)|}{u_{\delta,n}(x)}(|u_{n}(x)|-|u_{n}(y)|)(\varphi(x)-\varphi(y)) 
\end{align}
where in the last inequality we used the fact that
$$
\left(\frac{|u_{n}(x)|}{u_{\delta,n}(x)}-\frac{|u_{n}(y)|}{u_{\delta,n}(y)} \right) (|u_{n}(x)|-|u_{n}(y)|)\varphi(y)\geq 0
$$
because
$$
h(t)=\frac{t}{\sqrt{t^{2}+\delta^{2}}} \mbox{ is increasing for } t\geq 0 \quad \mbox{ and } \quad \varphi\geq 0 \mbox{ in }\R^{3}.
$$

Observing that
$$
\frac{|\frac{|u_{n}(x)|}{u_{\delta,n}(x)}(|u_{n}(x)|-|u_{n}(y)|)(\varphi(x)-\varphi(y))|}{|x-y|^{N+2s}}\leq \frac{||u_{n}(x)|-|u_{n}(y)||}{|x-y|^{\frac{3+2s}{2}}} \frac{|\varphi(x)-\varphi(y)|}{|x-y|^{\frac{3+2s}{2}}}\in L^{1}(\R^{6}),
$$
and $\frac{|u_{n}(x)|}{u_{\delta,n}(x)}\rightarrow 1$ a.e. in $\R^{3}$ as $\delta\rightarrow 0$,
we can use \eqref{alves1}, \eqref{alves2} and the Dominated Convergence Theorem to deduce that
\begin{align}\label{Kato2}
&\limsup_{\delta\rightarrow 0} \Re\left[\iint_{\R^{6}} \frac{(u_{n}(x)-u_{n}(y)e^{\imath A_{\e}(\frac{x+y}{2})\cdot (x-y)})}{|x-y|^{3+2s}} \left(\frac{\overline{u_{n}(x)}}{u_{\delta,n}(x)}\varphi(x)-\frac{\overline{u_{n}(y)}}{u_{\delta,n}(y)}\varphi(y)e^{-\imath A_{\e}(\frac{x+y}{2})\cdot (x-y)}  \right) dx dy\right] \nonumber\\
&\geq \limsup_{\delta\rightarrow 0} \iint_{\R^{6}} \frac{|u_{n}(x)|}{u_{\delta,n}(x)}(|u_{n}(x)|-|u_{n}(y)|)(\varphi(x)-\varphi(y)) \frac{dx dy}{|x-y|^{3+2s}} \nonumber\\
&=\iint_{\R^{6}} \frac{(|u_{n}(x)|-|u_{n}(y)|)(\varphi(x)-\varphi(y))}{|x-y|^{3+2s}} dx dy.
\end{align}
We can also see that the Dominated Convergence Theorem  (we recall that $\frac{|u_{n}|^{2}}{u_{\delta, n}}\leq |u_{n}|$, Fatou's Lemma and $\varphi\in C^{\infty}_{c}(\R^{3}, \R)$) yield 
\begin{equation}\label{Kato3}
\lim_{\delta\rightarrow 0} \int_{\R^{3}} V(\e x)\frac{|u_{n}|^{2}}{u_{\delta,n}}\varphi dx=\int_{\R^{3}} V(\e x)|u_{n}|\varphi dx\geq \int_{\R^{3}} V_{0}|u_{n}|\varphi dx
\end{equation}
\begin{equation}\label{KatoP}
\liminf_{\delta\rightarrow 0} \int_{\R^{3}} \phi_{|u_{n}|}^{t} \frac{|u_{n}|^{2}}{u_{\delta,n}}\varphi dx\geq \int_{\R^{3}} \phi_{|u|}^{t} |u|\varphi dx\geq 0
\end{equation}
and
\begin{equation}\label{Kato4}
\lim_{\delta\rightarrow 0}  \int_{\R^{3}} g(\e x, |u_{n}|^{2})\frac{|u_{n}|^{2}}{u_{\delta,n}}\varphi dx=\int_{\R^{3}} g(\e x, |u_{n}|^{2}) |u_{n}|\varphi dx.
\end{equation}
Taking into account \eqref{Kato1}, \eqref{Kato2}, \eqref{KatoP}, \eqref{Kato3} and \eqref{Kato4} we can deduce that
\begin{align*}
\iint_{\R^{6}} \frac{(|u_{n}(x)|-|u_{n}(y)|)(\varphi(x)-\varphi(y))}{|x-y|^{3+2s}} dx dy+\int_{\R^{3}} V_{0}|u_{n}|\varphi dx\leq 
\int_{\R^{3}} g(\e x, |u_{n}|^{2}) |u_{n}|\varphi dx
\end{align*}
for any $\varphi\in C^{\infty}_{c}(\R^{3}, \R)$ such that $\varphi\geq 0$, that is $|u_{n}|$ is a weak subsolution to \eqref{Kato0}.
 
Now, we note that $v_{n}=|u_{n}|(\cdot+\tilde{y}_{n})$ solves
\begin{equation}\label{Pkat}
(-\Delta)^{s} v_{n} + V_{0}v_{n}\leq g(\e_{n} x+\e_{n}\tilde{y}_{n}, v_{n}^{2})v_{n} \mbox{ in } \R^{3}. 
\end{equation}
Let us denote by $z_{n}\in H^{s}(\R^{3}, \R)$ the unique solution to
\begin{equation}\label{US}
(-\Delta)^{s} z_{n} + V_{0}z_{n}=g_{n} \mbox{ in } \R^{3},
\end{equation}
where
$$
g_{n}:=g(\e_{n} x+\e_{n}\tilde{y}_{n}, v_{n}^{2})v_{n}\in L^{r}(\R^{3}, \R) \quad \forall r\in [2, \infty].
$$
Since \eqref{UBu} yields $\|v_{n}\|_{L^{\infty}(\R^{3})}\leq C$ for all $n\in \mathbb{N}$, by interpolation we know that $v_{n}\rightarrow v$ strongly converges in $L^{r}(\R^{3}, \R)$ for all $r\in (2, \infty)$, for some $v\in L^{r}(\R^{3}, \R)$. From the growth assumptions on $f$, we have $g_{n}\rightarrow  f(v^{2})v$ in $L^{r}(\R^{3}, \R)$ and $\|g_{n}\|_{L^{\infty}(\R^{3})}\leq C$ for all $n\in \mathbb{N}$.
In view of \cite{FQT}, we know that $z_{n}=\mathcal{K}*g_{n}$, where $\mathcal{K}$ is the Bessel kernel, and proceeding as in \cite{AM}, we can infer that $|z_{n}(x)|\rightarrow 0$ as $|x|\rightarrow \infty$ uniformly with respect to $n\in \mathbb{N}$.
Since $v_{n}$ solves \eqref{Pkat} and $z_{n}$ verifies \eqref{US}, it is easy to use a comparison argument to deduce that $0\leq v_{n}\leq z_{n}$ a.e. in $\R^{3}$ and for all $n\in \mathbb{N}$. Therefore $v_{n}(x)\rightarrow 0$ as $|x|\rightarrow \infty$ uniformly with respect to $n\in \mathbb{N}$.
\end{proof}

\noindent
Now, we are ready to give the proof of Theorem \ref{thm1}.
\begin{proof}[Proof of Thorem \ref{thm1}]
In view of Lemma \ref{prop3.3}, we can find $(\tilde{y}_{n})\subset \R^{3}$ such that $\e_{n}\tilde{y}_{n}\rightarrow y_{0}$ for some $y_{0} \in \Lambda$ such that $V(y_{0})=V_{0}$. 
Then there is $r>0$ such that, for some subsequence still denoted by itself, it holds $B_{r}(\tilde{y}_{n})\subset \Lambda$ for all $n\in \mathbb{N}$.
Thus $B_{\frac{r}{\e_{n}}}(\tilde{y}_{n})\subset \Lambda_{\e_{n}}$ $n\in \mathbb{N}$, and we can deduce that 
$
\R^{3}\setminus \Lambda_{\e_{n}}\subset \R^{3} \setminus B_{\frac{r}{\e_{n}}}(\tilde{y}_{n}) \mbox{ for any } n\in \mathbb{N}.
$ 
By using Lemma \ref{moser}, we know that there exists $R>0$ such that 
$$
v_{n}(x)<a \mbox{ for } |x|\geq R, n\in \mathbb{N},
$$ 
where $v_{n}(x)=|u_{\e_{n}}|(x+ \tilde{y}_{n})$. 
Thus $|u_{\e_{n}}(x)|<a$ for any $x\in \R^{N}\setminus B_{R}(\tilde{y}_{n})$ and $n\in \mathbb{N}$. Then there exists $\nu \in \mathbb{N}$ such that for any $n\geq \nu$ and $r/\e_{n}>R$ it holds 
$$
\R^{3}\setminus \Lambda_{\e_{n}}\subset \R^{3} \setminus B_{\frac{r}{\e_{n}}}(\tilde{y}_{n})\subset \R^{3}\setminus B_{R}(\tilde{y}_{n}),
$$ 
which gives $|u_{\e_{n}}(x)|<a$ for any $x\in \R^{3}\setminus \Lambda_{\e_{n}}$ and $n\geq \nu$.  \\
Therefore, there exists $\e_{0}>0$ such that problem \eqref{Pe} admits a nontrivial solution $u_{\e}$ for all $\e\in (0, \e_{0})$. 
Setting $\hat{u}_{\e}(x)=u_{\e}(x/\e)$, we can see that $\hat{u}_{\e}$ is a solution to the original problem (\ref{P}). 
Finally, we investigate the behavior of the maximum points of  $|u_{\e_{n}}|$. By using $(g_1)$, there exists $\gamma\in (0,a)$ such that
\begin{align}\label{4.18HZ}
g(\e x, t^{2})t^{2}\leq \frac{V_{0}}{2}t^{2}, \mbox{ for all } x\in \R^{3}, |t|\leq \gamma.
\end{align}
Arguing as before, we can take $R>0$ such that
\begin{align}\label{4.19HZ}
\|u_{\e_{n}}\|_{L^{\infty}(B^{c}_{R}(\tilde{y}_{n}))}<\gamma.
\end{align}
Up to a subsequence, we may also assume that
\begin{align}\label{4.20HZ}
\|u_{\e_{n}}\|_{L^{\infty}(B_{R}(\tilde{y}_{n}))}\geq \gamma.
\end{align}
Indeed, if \eqref{4.20HZ} does not hold, we have $\|u_{\e_{n}}\|_{L^{\infty}(\R^{3})}< \gamma$, and by using $J_{\e_{n}}'(u_{\e_{n}})=0$, \eqref{4.18HZ} and Lemma \ref{DI} we can see that 
$$
[|u_{\e_{n}}|]^{2}+\int_{\R^{3}}V_{0}|u_{\e_{n}}|^{2}dx\leq \|u_{\e_{n}}\|^{2}_{\e_{n}}+\int_{\R^{3}} \phi_{|u_{\e_{n}}|}^{t}|u_{\e_{n}}|^{2}dx=\int_{\R^{3}} g_{\e_{n}}(x, |u_{\e_{n}}|^{2})|u_{\e_{n}}|^{2}\,dx\leq \frac{V_{0}}{2}\int_{\R^{3}}|u_{\e_{n}}|^{2}\, dx
$$
that is $\|u_{\e_{n}}\|_{H^{s}(\R^{3})}=0$ which is a contradiction. Therefore \eqref{4.20HZ} holds true.
In view of \eqref{4.19HZ} and \eqref{4.20HZ}, we can see that the maximum points $p_{n}$ of $|u_{\e_{n}}|$ belong to $B_{R}(\tilde{y}_{n})$, that is $p_{n}=\tilde{y}_{n}+q_{n}$ for some $q_{n}\in B_{R}$. 
Since $\hat{u}_{n}(x)=u_{\e_{n}}(x/\e_{n})$ is a solution to \eqref{P}, we can deduce that a maximum point $\eta_{\e_{n}}$ of $|\hat{u}_{n}|$ is of the type $\eta_{\e_{n}}=\e_{n}\tilde{y}_{n}+\e_{n}q_{n}$. Since $q_{n}\in B_{R}$, $\e_{n}\tilde{y}_{n}\rightarrow y_{0}$ and $V(y_{0})=V_{0}$, we can use the continuity of $V$ to deduce that
$$
\lim_{n\rightarrow \infty} V(\eta_{\e_{n}})=V_{0}.
$$
Finally, we prove the power decay estimate of $|\hat{u}_{n}|$. By applying Lemma $4.3$ in \cite{FQT}, we can find a function $w$ such that 
\begin{align}\label{HZ1}
0<w(x)\leq \frac{C}{1+|x|^{3+2s}},
\end{align}
and
\begin{align}\label{HZ2}
(-\Delta)^{s} w+\frac{V_{0}}{2}w\geq 0 \mbox{ in } \R^{3}\setminus B_{R_{1}} 
\end{align}
for some suitable $R_{1}>0$. Invoking Lemma \ref{moser}, we know that
we can that $v_{n}(x)\rightarrow 0$ as $|x|\rightarrow \infty$ uniformly in $n\in \mathbb{N}$, so 
we can find there exists $R_{2}>0$ such that
\begin{equation}\label{hzero}
h_{n}=g(\e_{n}x+\e_{n}\tilde{y}_{n}, v_{n}^{2})v_{n}\leq \frac{V_{0}}{2}v_{n}  \mbox{ in } B_{R_{2}}^{c}.
\end{equation}
Let $w_{n}$ be the unique solution to 
$$
(-\Delta)^{s}w_{n}+V_{0}w_{n}=h_{n} \mbox{ in } \R^{3}.
$$
Then $w_{n}(x)\rightarrow 0$ as $|x|\rightarrow \infty$ uniformly in $n\in \mathbb{N}$, and by comparison $0\leq v_{n}\leq w_{n}$ in $\R^{3}$. By using \eqref{hzero} we can see that
\begin{align*}
(-\Delta)^{s}w_{n}+\frac{V_{0}}{2}w_{n}=h_{n}-\frac{V_{0}}{2}w_{n}\leq 0 \mbox{ in } B_{R_{2}}^{c}.
\end{align*}
Set $R_{3}=\max\{R_{1}, R_{2}\}$ and we define
\begin{align}\label{HZ4}
a=\inf_{B_{R_{3}}} w>0 \mbox{ and } \tilde{w}_{n}=(b+1)w-a w_{n}.
\end{align}
where $b=\sup_{n\in \mathbb{N}} \|w_{n}\|_{L^{\infty}(\R^{3})}<\infty$. 
We aim to prove that
\begin{equation}\label{HZ5}
\tilde{w}_{n}\geq 0 \mbox{ in } \R^{3}.
\end{equation}
We begin observing that
\begin{align}
&\lim_{|x|\rightarrow \infty} \sup_{n\in \mathbb{N}}\tilde{w}_{n}(x)=0,  \label{HZ0N} \\
&\tilde{w}_{n}\geq ba+w-ba>0 \mbox{ in } B_{R_{3}} \label{HZ0},\\
&(-\Delta)^{s} \tilde{w}_{n}+\frac{V_{0}}{2}\tilde{w}_{n}\geq 0 \mbox{ in } \R^{3}\setminus B_{R_{3}} \label{HZ00}.
\end{align}
We assume by contradiction that there exists a sequence $(\bar{x}_{j, n})\subset \R^{3}$ such that 
\begin{align}\label{HZ6}
\inf_{x\in \R^{3}} \tilde{w}_{n}(x)=\lim_{j\rightarrow \infty} \tilde{w}_{n}(\bar{x}_{j, n})<0. 
\end{align}
Clearly, from (\ref{HZ0N}), it follows that $(\bar{x}_{j, n})$ is bounded, and, up to subsequence, we may assume that there exists $\bar{x}_{n}\in \R^{N}$ such that $\bar{x}_{j, n}\rightarrow \bar{x}_{n}$ as $j\rightarrow \infty$. 
Then (\ref{HZ6}) implies that
\begin{align}\label{HZ7}
\inf_{x\in \R^{3}} \tilde{w}_{n}(x)= \tilde{w}_{n}(\bar{x}_{n})<0.
\end{align}
By using the minimality of $\bar{x}_{n}$ and the representation formula for the fractional Laplacian \cite{DPV}, we obtain that 
\begin{align}\label{HZ8}
(-\Delta)^{s}\tilde{w}_{n}(\bar{x}_{n})=\frac{C(3, s)}{2} \int_{\R^{3}} \frac{2\tilde{w}_{n}(\bar{x}_{n})-\tilde{w}_{n}(\bar{x}_{n}+\xi)-\tilde{w}_{n}(\bar{x}_{n}-\xi)}{|\xi|^{3+2s}} d\xi\leq 0.
\end{align}
In view of (\ref{HZ0}) and (\ref{HZ6}), we have $\bar{x}_{n}\in \R^{3}\setminus B_{R_{3}}$, and 
by using (\ref{HZ7}) and (\ref{HZ8}), we can conclude that 
$$
(-\Delta)^{s} \tilde{w}_{n}(\bar{x}_{n})+\frac{V_{0}}{2}\tilde{w}_{n}(\bar{x}_{n})<0,
$$
which is impossible due to (\ref{HZ00}).
Therefore, (\ref{HZ5}) is true and by using (\ref{HZ1}) and $v_{n}\leq w_{n}$ we have
\begin{align*}
0\leq v_{n}(x)\leq w_{n}(x)\leq \frac{(b+1)}{a}w(x)\leq \frac{\tilde{C}}{1+|x|^{3+2s}} \mbox{ for all } n\in \mathbb{N}, x\in \R^{3},
\end{align*}
for some constant $\tilde{C}>0$. 
Taking in mind the definition of $v_{n}$, we can infer that  
\begin{align*}
|\hat{u}_{n}|(x)&=|u_{\e_{n}}|\left(\frac{x}{\e_{n}}\right)=v_{n}\left(\frac{x}{\e_{n}}-\tilde{y}_{n}\right) \\
&\leq \frac{\tilde{C}}{1+|\frac{x}{\e_{n}}-\tilde{y}_{\e_{n}}|^{3+2s}} \\
&=\frac{\tilde{C} \e_{n}^{3+2s}}{\e_{n}^{3+2s}+|x- \e_{n} \tilde{y}_{\e_{n}}|^{3+2s}} \\
&\leq \frac{\tilde{C} \e_{n}^{3+2s}}{\e_{n}^{3+2s}+|x-\eta_{\e_{n}}|^{3+2s}}.
\end{align*}

\end{proof}

\end{document}